\documentclass{article}
\usepackage{ijcai22}

\usepackage{xr}
\usepackage{times}

\usepackage{soul}
\usepackage{url}
\usepackage[hidelinks]{hyperref}
\usepackage[utf8]{inputenc}
\usepackage[small]{caption}
\usepackage{graphicx}
\usepackage{amsmath}
\usepackage{booktabs}
\usepackage{amsthm}

\usepackage{times}  % DO NOT CHANGE THIS
\usepackage{helvet} % DO NOT CHANGE THIS
\usepackage{courier}  % DO NOT CHANGE THIS

\usepackage{caption} % DO NOT CHANGE THIS AND DO NOT ADD ANY OPTIONS TO IT
\DeclareCaptionStyle{ruled}{labelfont=normalfont,labelsep=colon,strut=off} % DO NOT CHANGE THIS
\usepackage{amssymb}

\usepackage{color}

\usepackage{algpseudocode}

\newcommand{\E}{\mathop{\mathbb{E}}}
\newcommand{\Pois}{\mathrm{Pois}}
\def\calL{\mathcal{L}}
\def\calD{\mathcal{D}}
\def\N{\mathbb{N}}
\def\R{\mathbb{R}}
\def\ss{{s,s'}}

\newtheorem{lemma}{Lemma}

\newtheorem{definition}{Definition}
\newtheorem{theorem}{Theorem}
\newtheorem{proposition}{Proposition}

\newcommand{\citet}[1]{\citeauthor{#1}~\shortcite{#1}}

\usepackage{algorithm}

\newcommand{\vvspace}[1]{}

\title{Dynamic Car Dispatching and Pricing: Revenue and Fairness for Ridesharing Platforms}

\author{
Zishuo Zhao$^1$\and
Xi Chen$^2$\and
Xuefeng Zhang$^{1}$\And
Yuan Zhou$^3$\footnote{Corresponding Author}\\
\affiliations
$^1$University of Illinois Urbana-Champaign, 
$^2$New York University,\\
$^3$Yau Mathematical Sciences Center, Tsinghua University
\emails
{zishuoz2@illinois.edu},
{xc13@stern.nyu.edu},
{xuefeng8@illinois.edu},
{yuan-zhou@tsinghua.edu.cn}
}

\begin{document}
%\title{Dynamic Car Dispatching and Pricing: Revenue and Fairness for Ridesharing}
%\author{Zishuo Zhao, Xi Chen, Xuefeng Zhang, Yuan Zhou}
%\author{Submission \#1568}
\maketitle

%\vvspace{-1ex}

\begin{abstract}
%It is essential for a ridesharing platform to balance supply and demand and dynamically determine prices at a spatiotemporal scale. However,  a 

A major challenge for ridesharing platforms is to guarantee profit and fairness simultaneously, especially in the presence of misaligned incentives of drivers and riders. We focus on the  dispatching-pricing problem to maximize the total revenue while keeping both drivers and riders satisfied. We study the computational complexity of the problem, provide a novel two-phased pricing solution with revenue and fairness guarantees, extend it to stochastic settings and develop a dynamic (a.k.a., learning-while-doing) algorithm that actively collects data to learn the demand distribution during the scheduling process. We also conduct extensive experiments to demonstrate the effectiveness of our algorithms.

%    In particular, our mechanism first determines optimal prices charged from riders, and then reallocates the revenue in a driver-fair way. The decoupling of riders' prices and drivers' rewards allows more flexible schemes to optimize revenue while guaranteeing subgame-perfect and envy-free properties for both drivers and riders. In the first phase, we develop a maximum weighted flow model that theoretically yields maximum revenue for deterministic future orders under a regularity assumption, and also prove the NP-hardness of the revenue optimization problem in the general case. In the second phase, we use a potential-based optimization to guarantee subgame-perfect equilibrium and envy-freeness for drivers. We further extend to the learning setting to estimate the key model parameters on the fly, and adopt the Thompson sampling to balance the exploration and exploitation.  We illustrate the developed car dispatching and dynamic pricing algorithm by experimental studies.
\end{abstract}

\vvspace{-3ex}
\section{Introduction}

Ridesharing is a novel form of sharing economy that utilizes mobile apps to match drivers and riders to allow riders to take trips conveniently and make profits for drivers. Compared to traditional taxi platforms, ridesharing platforms enable riders to put orders on the system in advance of the trip for drivers to take, so that the system can optimally plan the rides to make it more efficient. Previous studies on planning algorithms for ridesharing platforms adopt a variety of methodologies including combinatorial optimization \cite{Bei2018AlgorithmsFT}, reinforcement learning \cite{www}, or both \cite{2020Ride}. However, planning trips only in the centralized way does not guarantee that each individual driver and rider has the incentive to obey the plan, which calls for efficient and fair pricing mechanisms so that following the plan will be ``happy'' for each party and maximize their utilities.

The pricing mechanism for taxi platforms depends on distance and waiting time, but it is too simple to either well represent the cost of drivers or match the supply and demand, which may result in dissatisfaction on both sides and lead to refusal of trips. For example, a rider wants to take an important trip with a short distance and a low price. However, there is a traffic jam and it may take a long time for the driver to cover the trip. This situation will create an opportunity cost that discourages the driver to accept the order. Were the charged price higher, the rider would probably not mind the slight increase of cost but the driver will be satisfied to accept the trip, which benefits both parties. However, we should be careful about the price adjustment: if two friends take the same trip, but at different prices, the one who takes the trip with a higher price may ``envy'' the other and will be dissatisfied with the platform. This issue may also apply to the drivers:  if two drivers initially at the same time and location are assigned different trips that earn different profits, the driver with lower profit would also be dissatisfied with the platform. Therefore, to make the platform satisfied by each agent, the algorithm should be ``envy-free'' (as in Definition~\ref{def:ef}). Another important property is ``subgame-perfect Nash equilibrium'', which means that each driver is assigned with a plan, following which he/she can get the best utility among all alternative actions given others' actions are fixed, so that no driver has the incentive to deviate from the plan (as in Definition \ref{def:sp}).

%\comment{Contribution:start with ``In this paper, we propose XXX. XXXX, Our contribute is three-fold: 1, 2, 3 }

In this paper, we propose a fairness-aware algorithmic framework for dynamic car dispatching and pricing, which consists of the following three-fold contributions:

\vvspace{-1ex}
\begin{enumerate}
    \item We study the computational complexity of the task of dispatching and pricing for total revenue maximization, propose a versatile generalized network flow model for the task, and provide theoretical guarantees (Section~\ref{sec:phase-1-deterministic}).
    \item We propose a novel two-phase pricing mechanism that decouples and sets different prices on drivers' and riders' sides, which can adapt to situations where the  drivers' and riders' interests misalign\footnote{Please see the illustrative example in Appendix \ref{appendix:example_two_phase}.} and guarantee fairness for both parties (Section~\ref{sec:phase-2-deterministic}).
    \item We consider the stochastic nature of ridesharing orders and study the online learning setting. We natural extend the model to the stochastic setting (Section~\ref{sec:stochastic-setting}), enabling the use of Thompson sampling-based algorithm to learn the valuation distributions from the partial information given by the riders' responses, and balance the exploration-exploitation trade-off (Appendix~\ref{sec:online-learning}).
\end{enumerate}
\vvspace{-1ex}
Finally, in Section~\ref{sec:experiments}, we perform extensive experimental evaluations of our assumptions and algorithms in the real-world datasets and demonstrate the effectiveness of our methods.

We have also shown that our algorithm runs in polynomial time. Please refer to Section~\ref{app:complexity} for detailed complexity analysis.
%\xnote{Defer the comparison to other work after we introduce our contribution}

%\comment{Add pure ridesharing, yjp, qin zhiwei}
\vvspace{-1ex}
\paragraph{Related Works.} There are several related works in the existing literature. \citet{Bei2018AlgorithmsFT,2020Ride,wang2018deep} study how to dispatch the drivers efficiently in a centralized way, and 
%\citet{bandara2009multi,hrnvcivr2015ridesharing,www} 
\citet{hrnvcivr2015ridesharing,www} 
study the dispatching problem via multiagent systems, but they do not consider pricing which is essential for the application in platforms. \citet{queue_pricing} study optimal pricing via queue theory, but they assume a single location, which is too simple for application. \citet{wild_goose} discuss the phenomenon of ``wild goose chase'' in which drivers spent most time driving to take a distant order in unbalanced supply and demand, and propose the method of adjusting price to avoid its detriment to efficiency, but they do not consider fairness. \citet{2019Spatial} look into the effects of pricing to supply-demand balance, revenue and consumers' surplus, but adopt an over-simplified model of $n$ pairwise equidistant locations, which is not even geometrically possible for large $n$. %\xnote{add citations. Point out how they are different from us.}
\citet{yan2018dynamic} also provide an algorithm for dynamic matching and pricing, but the matching and pricing algorithms are decoupled, making the performance suboptimal.
In particular,  the very recent work \cite{pricing1}, which shares a similar motivation as our work, studies how to maximize social welfare, i.e. the summation of riders' valuations minus drivers' costs among all trips, via an bidding-based dispatching and pricing algorithm.  
%\xnote{clarify social welfare here} 
In that paper, each rider should bid a maximally acceptable price for them, and the truthful mechanism guarantees that it is in each rider's interest to report their true valuation. However, there are some gaps from their mechanism to the reality. First of all,  it is not practical for riders to bid their valuation like an auction. Second, the mechanism maximizes total social welfare, not drivers' revenue, but ride-sharing platforms are indeed interested in their profits.  Also, it assumes that all future orders is known at the beginning, which is not realistic. In contrast, our algorithm optimizes the total revenue via dynamically learning the order distribution from the riders' responses on our carefully designed prices.

%Therefore, it motivates us to design a practical mechanism that maximizes profits, ensures fairness among both riders and drivers, and can be easily implemented on real ride-sharing platforms.

%The main content of this paper is divided into four parts. In the first and second parts, we assume that all future orders, consisting of origin, destination and valuation, are known beforehand, and design a two-phase algorithm to maximize drivers' collective revenue and appropriately distribute it to individual drivers, subject to fairness constraints. In the third part, we consider the stochastic model that only the \textit{distribution} of orders and valuations are known, and reduce the problem to the deterministic model, so that we can optimize the total \textit{expected} revenue, and discuss about methods to estimate the distribution of valuations via adaptively offering different prices and observe riders' responses (accept or decline). In the fourth part, we perform simulated experiments to show that our assumptions are suitable for real-world applications, evaluate the performance of our algorithms on DiDi dataset generated from real-world data, and use this dataset to demonstrate the effectiveness of fair reward re-allocation.

%\xnote{might need to introduce some notations here, e.g., $[T]=\{1,\ldots, T\}$.}
%\xnote{I suggest use the word ``platform'' not ``system''.}

\vvspace{-1.5ex}
\section{Preliminaries}
\vvspace{-1ex}
%For an positive integer $n$, we adopt the notation $[n]=\{1,2\cdots,n\}$.  
We assume the service zone is divided into a family $L$ of discrete locations, and the planning horizon is a family $T$ of discrete time slots. Therefore, there are $|L| \cdot |T|$ spatiotemporal \emph{states}, denoted by $S = L\times T$.

We also assume that the travelling time from one state $s = (l, t) \in S$ to another location $s'$ is deterministically defined by the known function $\delta(l, l', t) \in \mathbb{Z}^+$. We call each pair of the spatiotemporal states $(s, s')$ a spatiotemporal \emph{arc}. For each $s = (l, t)$ and $s' = (l', t')$, we say the arc $(s, s')$ is \emph{admissible} if $t' \geq t + \delta(l, l', t)$. We denote by $Q$ the set of all admissible arcs.

%Each state $s\in S$ can be denoted as $(l,t)$, in which $l$ is its location and $t$ is the time stamp. If some rider would travel from location $l$ to $l'$, starting at time $t$, we assume that the traveling time is deterministic as a function $\delta: L \times L \times T \to \mathbb{N}$, so that the trip would start at state $s=(l,t)$ and finish at state $s'=(l',t+\delta(l,l',t))$. 
 
%Each arc $(s, s')$ is associated by known deterministic cost $c(s,s')$. Staying at the same location for some time is also treated as a dummy arc. Not all spatiotemporal arcs are reachable, as we cannot travel backwards in time and cannot travel too far within limited time, so we define $Q\subseteq S\times S$ as a state reachability set such that $(s,s')\in Q$ if and only if the arc from $s$ to $s'$ is reachable. From the trip duration function $\delta$, we can generate the $Q$ with the rules:

%\begin{itemize}
%    \item $((l,t),(l,t+1)) \in Q$,  $\forall l\in L, t\in T.$
%    \item If $t+\delta(l,l',t)\in T$, then$((l,t), (l',t+\delta(l,l',t))) \in Q$, $\forall l,l' \in L, t\in T.$
%    \item If $(s_1,s_2)\in Q$ and $(s_2,s_3)\in Q$, then $(s_1,s_3) \in Q$. 
%\end{itemize}

Each admissible arc $(s, s') \in Q$ is associated with a known deterministic cost $c(s, s')$ which is incurred to any driver that drives along this arc. The order of the $i$-th rider is described by an admissible arc $(s_i, s_i') \in Q$ and a valuation $v_i$ which is the maximum amount the rider would like to pay for the ride. Since $v_i$ is not revealed to the ridesharing platform, we call $o_i = (s_i, s_i', v_i)$ the $i$-th \emph{latent order}, and denote $R = \{o_i, \forall i\}$ the set of latent orders.

%For each rider $i$, they would like to travel by arc $A_i = (s_i,s_i')\in Q$ with valuation $v_i$ that they would be willing to pay. We call such an order with the valuation as a \textbf{latent order}, and denote by $R=\{o_i = (s_i,s'_i,v_i)\}$ the set of latent orders. It is the platform's decision to offer a price $p_{(i)}$ for the trip, and the rider would take the ride if and only if $p_{(i)} \le v_i$. %\comment{Definition first, then explain} If different riders $i,i'$ would take the same arc $A_i=A_i'$ but with different valuations $v_i\neq v_i'$, we should be envy-free to set the same price for them to ensure the fairness. Therefore, the pricing mechanism should have the \textit{anonymous} property that only depends on the arc, so we can define a rider-side pricing function $p:S\times S \to [0,+\infty)$. 

The task of the scheduling algorithm for the ride-sharing platform involves the decision of a rider-side pricing function $p: S\times S \to [0,+\infty)$ (which has to be independent of the rider to ensure envy-freeness). For each rider $i$ with latent order $o_i=(s_i, s_i', v_i)$, the scheduling algorithm offers the price $p(s_i, s_i')$.
%and dispatches a driver to pick up the rider. 
The rider only accepts the offer if $v_i \geq p(s_i, s_i')$ in which case the platform receives $p(s_i, s_i')$ as income. Serving the order also incurs the driver cost according to $c(\cdot, \cdot)$ along the arcs. After all trips, the drivers will leave the platform. 

The first goal of the scheduling algorithm is to maximize the total revenue which is defined to be the total income (collected from the riders) minus the total cost (incurred by the drivers). Then, the second goal of the scheduling algorithm is to compute the driver-side payment function $y: S\times S \to [0,+\infty)$ to distribute the income to the drivers in a subgame-perfect and envy-free manner. (Note that the payment function also has to be independent of the drivers to ensure envy-freeness.)

%, so that the price charged from riders only depends on the spatiotemporal arc, which is a sufficient and necessary condition of rider-side envy-freeness. In similar sense, the payment to a driver who travels by an arc $(s,s')$ should also only depend on the arc itself, denoted as a driver-side pricing function .

%Each driver $j$ begins at state $x_j$ and can travel by a series of reachable arcs. If a rider is picked up on an arc from $s$ to $s'$, the platform would get a reward of $p(s,s')$, and for driver-side fairness, no matter if a rider is picked up, the platform should pay $r(s,s')$ to the driver. For each reachable arc $(s,s')\subseteq{Q}$, travelling by it also incurs a cost of $c(s,s')$ to the driver. After all trips drivers will leave the platform, and leaving at different states may incur different costs, denoted as $c_{out}(\cdot)$. We define the total revenue as the total collected reward minus the total drivers' cost, which is the objective we are to maximize under fairness constraints.

The above-described scheduling problem involves the complex optimization of multiple sets of decision variables. The unknown latent order set introduces further challenges to the task. To approach this complex problem, we will first consider the deterministic setting where the latent order set $R$ is fully revealed to the scheduling algorithm, and the scheduling problem becomes a pure static optimization task. Then, we assume that $R$ is drawn from a latent distribution, and design an online learning algorithm that simultaneously learns the latent distribution and optimizes the total revenue.

%We assume that $L,T,S,Q,c,c_{out}$ are known. In the deterministic setting, we also assume $R$ is known, and we should plan the routes of drivers and also decide on $p,r$ to optimize total revenue while guaranteeing fairness on both riders and drivers. In the learning setting, we assume $R$ obeys a latent distribution and the platform has to learn the distribution online while running.

In the following two sections, we describe each phase of the problem with more details and mathematical rigor, and propose our algorithms to achieve the optimal policy.

\vvspace{-1.5ex}

\section{Phase 1 of the Deterministic Setting: Maximum Revenue Car Dispatching} \label{sec:phase-1-deterministic}

\vvspace{-1ex}

In this section, we introduce our algorithm to the maximum revenue car dispatching problem in the deterministic setting (i.e., when the set of latent orders $R$ is known to the platform).  For convenience, we first introduce the following non-linearly weighted circulation (NLWC) problem, and the maximum revenue car dispatching problem can be formulated based on NLWC definition.

%\subsection{Problem Formulation} %\comment{more explicit ``Definition x.x:''}}

%\subsubsection{The Non-linear integer flow problem}
%Our general scheme is divided into 2 phases: rider-side and driver-side.

%\subsubsection{Rider-side: Maximum Revenue Car Dispatching}

%In this part we formulate the Maximum Revenue Car Dispatching problem. For convenience, we introduce the NLWC (non-linear integer flow) problem to which the revenue maximization of the deterministic setting can be immediately transformed to.  We will show the expected revenue maximization in the learning setting can also be transform to the NLWC problem, with an additional property that Regularity is obeyed in experiments, and the detail is in part ``The Learning Setting''.
\vvspace{-0.5ex}
\begin{definition}\label{def:nlwc}
In the \emph{non-linearly weighted circulation (NLWC)} problem, there is a directed graph $G = (V, E)$. For each directed edge $e \in E$, we associate it with the \emph{flow lower bound} $\ell(e)$, the \emph{flow upper bound} $u(e)$ and the \emph{reward function} $r(\cdot ; e) : \mathbb{N} \to \mathbb{R}$. The goal is to find a flow $f : E \to \mathbb{N}$ so that $f$ satisfies the flow upper and lower bounds (i.e., $\ell(e) \leq f(e) \leq u(e), \forall e \in E$) and flow conservation (i.e., $\sum_{e \text{~going out of~} s} f(e) = \sum_{e \text{~going into~} s} f(e), \forall s \in V$), and the total reward $\sum_{e \in E} r(f(e); e)$ is maximized. 
%is similar to the canonical maximum weighted flow problem in the integer setting, while the only difference is that when an edge $e(s,s')$ has capacity more than 1 and we arrange a flow of amount $x$ along $e$, its reward is a function $r_e(x):\mathbb{N}\to \mathbb{R}$ instead of $x$ multiplied by its weight. Each instance of NLWC problem can be depicted with a graph $G=(V,E)$, in which $V$ is the vertex set, $E$ is the edge set, and each edge in NLWC is formalized as $(s,s',c_-,c,w(\cdot))$, where $(s,s')$ is a directed edge, $c_-$ is its minimum flow, $c$ is its capacity and $w(\cdot)$ is its reward function.
\end{definition}
\vvspace{-0.5ex}
Observe that when the reward functions are linear (i.e., $r(x; e) = w(e) \cdot x$), the NLWC problem becomes the canonical minimum cost circulation problem, which admits a polynomial time algorithm \cite{tardos1985strongly} (with the signs of the linear coefficients flipped). 

With the formulation of the NLWC problem in place, we are ready  to describe our maximum revenue car dispatching problem in the deterministic setting. Here, we assume that the platform knows all the riders' information; i.e., for each rider $i$, we know that his/her latent order $o_i = (s_i, s_i', v_i)$. Based on this information, for each arc $(s,s')\in Q$, we calculate the number of latent orders following the arc and denote it by $o(s,s')$; then, for each $1\le i \le o(s,s')$, we define $v_i(s,s')$ to be the $i$-th largest valuation among all latent orders following $(s, s')$. Note that if the platform plans to accept $k$ orders on the arc $(s, s')$, to maximize the income, the price should be set as $p(s,s') = v_k(s,s')$, and the total income generated from this arc becomes $k \cdot v_k(s,s')$. 

In light of the discussion above, we will construct a directed graph $(V_0, E_0)$ so that the maximum revenue car dispatching problem becomes calculating NLWC on the graph, where the flow along each arc indicates the number of drivers the platform plans to dispatch. 

We first let the vertex set $V_0=S\cup \{I,O\}$ where $I$ is the artificial source and $O$ is the artificial sink; together, a directed edge $e_{O,I}$ that goes from $O$ to $I$ is set up with $\ell(e_{O,I}) = 0$ flow lower bound and $u(e_{O,I}) = +\infty$ flow upper bound and the constant-zero reward function: $r(\cdot; e_{O,I}) \equiv 0$. We then set up the following sets of edges.
\begin{itemize}
\item (\textbf{Initialize drivers.}) For each spatiotemporal state $s$ with $n_s$ initial drivers, we set up a directed edge $e_{I,s}$ going from $I$ to $s$ with both flow upper and lower bounds equal to $\ell(e_{I,s}) = u(e_{I,s}) = n_s$, and the constant-zero reward function $r(\cdot ; e_{I,s}) \equiv 0$. The flow $f(e_{I, s})$ represents the number of drivers to start working from the state $s$.

\item (\textbf{Leaving drivers.}) For any spatiotemporal state $s$, we set up a directed edge $e_{s,O}$ going from $s$ to $O$ with $\ell(e_{s,O}) = 0$ lower bound, $u(e_{s,O}) = +\infty$ upper bound, and the 
constant-zero
reward function $r(\cdot ; e_{s,O}) \equiv 0$. The flow $f(e_{s, O})$ represents the number of drivers to leave the system at the state $s$.

\item (\textbf{Driving without a rider.}) For any admissible arc $(s, s') \in Q$, we set up a directed edge $e^{(\mathrm{o})}_{s,s'}$ going from $s$ to $s'$ with $\ell(e^{(\mathrm{o})}_{s,s'}) = 0$ lower bound, $u(e^{(\mathrm{o})}_{s,s'}) = +\infty$ upper bound.
The flow $f = f(e^{(\mathrm{o})}_{s,s'})$ represents the number of drivers to drive through the arc $(s, s')$ without carrying a rider. Therefore, we set up the linear reward function $r(f; e^{(\mathrm{o})}_{s,s'}) = -c(s, s') \cdot f$.

\item (\textbf{Driving with a rider.}) For any admissible arc $(s, s') \in Q$, we set up a directed edge $e^{(\mathrm{w})}_{s,s'}$ going from $s$ to $s'$ with $\ell(e^{(\mathrm{w})}_{s,s'}) = 0$ lower bound, $u(e^{(\mathrm{w})}_{s,s'}) = o(s, s')$ upper bound. The flow $f = f(e^{(\mathrm{w})}_{s,s'})$ represents the number of drivers to drive through the arc $(s, s')$ with a rider. Therefore, we define the non-linear reward function $r(f; e^{(\mathrm{w})}_{s,s'}) = [v_f(s, s') - c(s, s')] \cdot f$.
\end{itemize}

Given $(V_0, E_0)$, the \emph{\underline{maximum revenue car dispatching}} problem in the deterministic setting is equivalent to finding the optimal solution to NLWC on the directed graph $(V_0, E_0)$. Formally, we directly have the proposition below.

\begin{proposition} \label{prop:mrcd}
Let $f^*$ be the optimal solution to NLWC on the directed graph $(V_0, E_0)$. To achieve the maximum revenue in the car dispatching task, the platform may direct the drivers to drive with/without carrying a rider or leave the platform based on the flow value on the corresponding sets of edges. The total weight of $f^*$ is the maximum revenue the platform may collect.
\end{proposition}

Proposition~\ref{prop:mrcd} also enables us to design the \emph{routing plan} for each individual driver based on the NLWC solution. Formally, a \emph{route} $A = (a_1,a_2,\dots,a_z)$ is a sequence of spatiotemporal arcs such that the ending state of each arc $a_i$ is the same as the beginning state of the next arc $a_{i+1}$ (for all $i \in \{1, 2, \dots, z - 1\}$). At each time step and for each driver $q$, the \emph{routing plan} $A_q$ is just a route which starts at the driver's current state.

While the general NLWC problem is computationally intractable, the maximum revenue car dispatching problem, as a special case of NLWC, is unfortunately not easier. Formally, we present the following negative result for the maximum revenue car dispatching problem. The proof of Theorem~\ref{thm:nphard} is deferred to Appendix~{\ref{appendix:proof:nphard}}. Note that since maximum revenue car dispatching is a special case, we may not directly use the NP-Hardness proof of NLWC, and have to design a new hardness instance instead.

\begin{theorem}
The maximum revenue car dispatching problem, even in the deterministic setting, is NP-hard.\label{thm:nphard}
\end{theorem}

On the positive side, we propose a natural \emph{regularity condition} in Definition~\ref{def:regularity}. We will show that when the condition is satisfied, the maximum revenue car dispatching problem can be solved in polynomial time.

%However, if we add a condition of Regularity (Definition \ref{def:regularity}), which is observed to hold in the Learning Setup, the NLWC problem becomes convex and can be transformed to canonical minimum cost circulation problem. Since the minimum cost circulation is totally unimodular\cite{}, as long as the constraints are integral, we are ensured to compute an integral optimal solution.

\begin{definition}[Regularity]\label{def:regularity}
We say that a maximum revenue car dispatching problem instance satisfies the \emph{regularity condition} if for each admissible spatiotemporal arc $(s,s')$, and each $k \in \{1, 2, \dots, o(s,s')\}$, the sequence $v'_k(s,s')$ is monotonically non-increasing with $k$, where we define
\begin{small}
\begin{align*}
v'_k(s,s') := \left\{
\begin{array}{ll}
     v_1(s, s') & (k = 1)\\
     k \cdot  v_k(s,s')-(k-1)v_{k-1}(s,s') & (k\geq 2)
\end{array}
\right. .
\end{align*}
\end{small}
\end{definition}

In the definition, $v'_k(s,s')$ can be interpreted as the \emph{marginal reward} of accepting the $k$-th highest price order on arc $(s,s')$. The regularity condition then requires that the marginal reward sequence is not increasing with the increasing number of accepted orders on any arc, which is a standard assumption in economics literature (see, e.g., \cite{marginal1,marginal2,marginal3}). Indeed, in our empirical evaluation, we verify that the regularity condition holds in the real-world data. 
%\xnote{I do not find such a verification in the experiment. Please delete this if we cannot provide such verification}

We now present our edge decomposition algorithm (details in Algorithm~\ref{alg:main_algo}) for the maximum revenue car dispatching problem. At a higher level, Algorithm~\ref{alg:main_algo} first manages to decompose each non-linear directed edge in $(V_0, E_0)$ to a family of edges with linear costs and creates a minimum linear-cost circulation problem instance $(V_0, \tilde{E}, \tilde{\ell}, \tilde{u}, -\tilde{w})$, then invokes the existing polynomial-time time algorithm for the minimum linear-cost circulation problem, and finally aggregates the flows in each family to construct the optimal solution to the original problem. 

\setlength{\textfloatsep}{8pt}
\begin{algorithm}[htb]
\caption{The Edge Decomposition Algorithm}
\label{alg:main_algo}
\begin{algorithmic}[1]
%\Procedure{Euclid}{$a,b$}\Comment{The g.c.d. of a and b}
%\State Construct new vertices $I,O$.
%\State $V=S\cup\{I,O\}$
%\State $E=\emptyset$
\State Construct the NLWC instance $(V_0,E_0,\ell,u,r)$;
%\For{$s: S$}
%    \State $n=\#$ drivers initially at state $s$
%    \State $E=E\cup\{(I,s,n,n,0),~(s,O,0,+\infty,0)\}$
%    \For{$s': S$}
%        \If{$(s,s')$ is a reachable path}
%            \State $E=E\cup\{(s,s',0,+\infty,-c(s,s')\}$
%            \State $o(s,s')=$number of orders on path $(s,s')$
%            \For{$k:[o(s,s')]$}
%                \State $e_k(s,s')=(s,s',0,1,v_k'(s,s')_{+}-c(s,s'))$
%                \State $E=E\cup\{e_k(s,s'\}$
%            \EndFor
%        \EndIf
%    \EndFor
%\EndFor

\State $E_1 \leftarrow \{e_{s,s'}^{(\mathrm{w})}\in E_0\}$, $E_2 \leftarrow E_0 - E_1$; $\tilde{E} \leftarrow \emptyset$;

\For {$e_{s,s'}^{(\mathrm{w})}\in E_1$} \label{line:alg-main-3}
    \For {$i\in \{1, 2, \dots, o(s,s')\}$}
        \State $\tilde{E} \leftarrow \tilde{E} \cup e_{s,s'}^{(\mathrm{w},i)}$; $(\tilde{\ell} (e_{s,s'}^{(\mathrm{w},i)}), \tilde{u} (e_{s,s'}^{(\mathrm{w},i)})) \leftarrow (0,1)$;

        \State $w(e_{s,s'}^{(\mathrm{w},i)})
        %\leftarrow v_i'(s,s')-c(s,s')$;
        \leftarrow r(i;e_{s,s'}^{(\mathrm{w})}) - r(i-1;e_{s,s'}^{(\mathrm{w})})$
    \EndFor
\EndFor \label{line:alg-main-8}

\For {$e\in E_2$}
    \State $\tilde{E} \leftarrow \tilde{E} \cup e$; $\left(\tilde{\ell} (e),\tilde{u} (e)\right) \leftarrow (\ell(e),u(e))$;
    \State $\tilde{w}(e) \leftarrow \left\{
\begin{array}{ll}
     -c(s, s') & (\text{if both~}s,s'\in S)\\
     %-c_{out}(s) & (\text{if~}s'= O) \\
     0 & (\text{otherwise})
\end{array}\right.$;
\EndFor

\State Invoke the polynomial-time algorithm \cite{tardos1985strongly} to compute the minimum cost circulation of $(V_0,\tilde{E}, \tilde{\ell}, \tilde{u}, -\tilde{w})$ where $-\tilde{w}$ is the coefficient function of the linear costs, denote the optimal flow  by $\tilde{f}$;

\For{$e \in E_0$}
\State {\bf if} $e = e_{s,s'}^{(\mathrm{w})} \in E_1$ {\bf then} $f(e) \leftarrow \sum_{i} \tilde{f}(e_{s,s'}^{(\mathrm{w},i)})$;
\State ~~~~~~~~~~~~~~~~~~~~~~~~~~~~{\bf else} $f(e) \leftarrow \tilde{f}(e)$;
\EndFor

\State {\bf return} $f$;
\end{algorithmic}
\vvspace{-1ex}
\end{algorithm}

In Algorithm~\ref{alg:main_algo}, the edge set $E_1$ denotes the edges corresponding to ``driving with a rider'' and  $E_2$  the rest of the edges. We also observe that the only non-linear edges are the ones to drive with a rider (in $E_1$), while the rest edges (in $E_2$) already have linear costs. For the edges in $E_1$, the algorithm decomposes them from Line~\ref{line:alg-main-3} to Line~\ref{line:alg-main-8}: since the flow on each edge in $E_1$ represents the amount of the rider orders accepted along the corresponding spatiotemporal arc, the algorithm assigns each decomposed edge with unitary capacity, and the corresponding flow represents an additional order to be accepted along the arc, and naturally the weight function is defined based on the marginal reward function $v_k'(\cdot, \cdot)$. Also note that the algorithm always returns an integral flow because of the integrality property of the minimum linear-cost circulation problem. Regarding the theoretical guarantee of Algorithm~\ref{alg:main_algo}, we prove the following theorem:
%in Appendix~{\ref{appendix:proof_reg_equiv}}.

\begin{theorem} \label{thm:reg_equiv}
Algorithm~\ref{alg:main_algo} runs in polynomial time, and when the regularity condition is met, the returned flow $f$ achieves the maximum revenue of the car dispatching problem on the directed graph $(V_0, E_0)$.
\end{theorem}

\begin{proof}
We only need to prove that in the NLWC problem with regularity, each non-linear edge in $E_1$ with finite capacity can be substituted by a finite number of linear edges.

Consider an edge $e = e_\ss^{(\mathrm{w})}\in E_1$, then $\ell(e)=0$. Then, for each $i\in N$ s.t. $1 \le i\le u(e)$, we add to $\tilde{E}$ an linear edge $e_i(s,s',0,1,w(i)-w(i-1))$. Since $r(i;e_{s,s'}^{(\mathrm{w})}) - r(i-1;e_{s,s'}^{(\mathrm{w})})$ decreases with $i$, when we should put $t$ amount of flow from $s,s'$ in $G_1$, the optimal plan is to saturate edges $e_\ss^{(\mathrm{w},1)},e_\ss^{(\mathrm{w},2)},\cdots,e_\ss^{(\mathrm{w},t)}$, with total reward $r(t;e_{s,s'}^{(\mathrm{w})})$, identical to the NLWC model.

Therefore, we realize the same edge-reward function as the NLWC model with a minimum cost circulation model. While the Maximum Revenue Car Dispatching problem needs integer solutions, from the total unimodularity property of the minimum cost circulation problem, it is guaranteed that our algorithm outputs an integer basic solution. Therefore, we can indeed solve regular Maximum Revenue Car Dispatching via the minimum cost circulation problem.
\end{proof}

We also remark that even when in the general scenario (without the regularity condition), a simple variation of Algorithm~\ref{alg:main_algo} also serves as a good approximation to the optimal solution. It virtually approximates the edge reward function by its concave envelope to ``iron'' it to a concave function \cite{chawla2007algorithmic}. Please refer to Appendix~{\ref{appendix:algo:non_regular}} for details.

\vvspace{-1.5ex}

\section{Phase 2 of the Deterministic Setting: Fair Reward Re-allocation to Drivers} \label{sec:phase-2-deterministic}

Recall that in Phase 1 we have found the maximum revenue that can be achieved by any dispatching plan in the deterministic setting. Along the way, we have also figured out how many drivers are needed for a spatiotemporal arc $(s, s') \in Q$ with a rider (namely $f(e_{s,s'}^{(\mathrm{w})})$) and without carrying a rider (namely $f(e_{s,s'}^{(\mathrm{o})})$). For convenience, we define $F(s, s') := f(e_{s,s'}^{(\mathrm{w})}) + f(e_{s,s'}^{(\mathrm{o})})$ to be the total number of drivers we plan to dispatch along the arc $(s, s')$. In this section, we develop methods to figure out the fair payment scheme $y : S \times S \to [0, +\infty)$ for driving along each spatiotemporal arc to ensure that the drivers are well incentivized to cooperate with the platform and execute the optimal-revenue dispatching plan. 
%\comment{Additionally, the platform pays $c_{out}(s)$ to drivers who leave at state $s$.} 
Formally,
we define the fairness condition as follows.

%In this section, we introduce our algorithm for the fair reward reallocation to drivers in the deterministic setting.  Recall that $y(s,s')$  is the driver-side payment function (i.e., the payment to each driver traveling from $s$ to $s'$). For arc $a=(s,s')$, we define $F(a)$ as the total flow of all edges on a arc $a$, i.e.
%\begin{equation}
%    F(s,s') = \sum_{e\in E \textup{~is on arc} (s,s')} f(e).
%\end{equation}
%(i.e., the total number of fulfilled orders on $a$). 
%Specifically, we will establish a reward reallocation mechanism satisfying the conditions,

%\xnote{We need to explain this condition $\sum_{a\in S^2}r(F(a);a) = \sum_{a\in S^2} (y(a)-c(a))F(a)$}

\begin{definition}[Fair re-allocation]\label{def:fair}
A re-allocation scheme is \emph{fair} if and only if following conditions are satisfied:
\begin{itemize}
\item \underline{Budget-balance.} Let $\mathcal{I}$ be the total income collected from the riders. This should also be the exact amount to be distributed to the drivers.\footnote{We omit the amount that the platform would like to keep for profit, which can be easily added to the constraint w.l.o.g.} Formally, it is required that
$\sum_{(s, s') \in Q} y(s, s') \cdot F(s, s') = \mathcal{I}$.
\item \underline{Individual-rationality.} For each arc driven, the payment should be at least the cost; i.e., for each $(s, s') \in Q$ so that $F(s, s') > 0$, it is required that $y(s, s') \geq c(s, s')$.
\item \underline{Subgame-perfectness.} This is formally defined soon in Definition~\ref{def:sp} which, together with the non-negative producer surplus condition, makes sure that the drivers do not have the incentive to refuse and deviate from the dispatching plan.
\item \underline{Envy-freeness.} This is formally defined in Definition~\ref{def:ef} which makes sure that the drivers do not complain that the dispatching plan is more favorable to others than themselves. 
\end{itemize}
\end{definition}

Note that we need to define subgame-perfectness and envy-freeness in details. Before doing this, we need to introduce a few new notations and definitions.

We will model the drivers' behavior as an \emph{extensive game} \cite{glazer1996extensive}, where, at each state, each driver has the freedom to choose any route starting from the current state. At any time step, let $A_q$ denote the routing plan given by the platform for the driver $q$, let $\mathcal{A} := \{A_1, A_2, \dots\}$ denote the set of routing plans for all drivers, and let $A_{-q} := \mathcal{A} \setminus \{A_q\}$. For each driver $q$, let $u_q(\mathcal{A})$ denote the utility (i.e., net profit) of driver $q$ if all drivers follow the routing plan $\mathcal{A}$. In particular, we have that $u_q(\mathcal{A}) = \sum_{(s, s') \in A_q} (y(s, s') - c(s, s'))$.

The subgame-perfectness condition requires that given reward re-allocation scheme and the set of routing plans for all drivers by the platform, any driver $q$ does not have the incentive to deviate from the routing plan given to him/her. Formally, we make the following definition.

\begin{definition}[Subgame-perfectness]\label{def:sp}
A reward re-allocation scheme is \emph{subgame-perfect} if at any time step, let $\mathcal{A} := \{A_1, A_2, \dots\}$ be the routing plans decided by the platform, and for any driver $q$, and for each route $A_q'$ sharing the same starting state as $A_q$, it holds that
$u_q(A_q,A_{-q}) \ge u_q(A'_q,A_{-q})$.
\end{definition}

Note that in game theory, a subgame-perfect Nash equilibrium in a extensive game is a strategy profile for the agents such that at any point of the game, the agents' strategies form a Nash equilibrium for the continuation of the game. Definition~\ref{def:sp} requires that reward re-allocation scheme makes sure that the routing plan given by Proposition~\ref{prop:mrcd} is a subgame-perfect Nash equilibrium.

We would also like to make sure that each driver does not feel comparably inferior than others at the same state. Formally, we define the envy-freeness condition as follows.
\vvspace{-0.5ex}
\begin{definition}[Envy-freeness]\label{def:ef}
A reward re-allocation scheme is \emph{envy-free} if at any time step, let $\mathcal{A} := \{A_1, A_2, \dots\}$ be the routing plans decided by the platform, and for any two drivers $q$ and $q'$ staying at the same state, it holds that
$u_q(\mathcal{A}) = u_{q'}(\mathcal{A})$.
\end{definition}

Now we have completed the formal definition of a fair re-allocation scheme. The following lemma provides an elegant characterization of all fair re-allocation schemes and enables us to find such schemes only among the potential-based re-allocation algorithms. The proof of Lemma~\ref{lem:potential} can be found in Appendix~{\ref{appendix:proof_lem_potential}}.

%Actually we can embed the driver-side fairness and the budget-balance constraints into a potential-based framework, as shown in Lemma \ref{lem:potential}.
\vvspace{-1ex}
\begin{lemma}
    \label{lem:potential}

Given a routing plan $\mathcal{A}$, a reward re-allocation is fair if and only if there exists a corresponding \emph{potential function} $P:S\to \mathbb{R}^{\geq 0}$ such that
\vvspace{-1ex}
\begin{enumerate}
\item For any $s \in S$ where $\mathcal{A}$ directs at least one driver to leave at state $s$ (we call such states the \emph{terminal states}), it holds that $P(s)=0$.
\item $\forall (s,s')\in Q$, $y(s,s')- c(s,s') \le P(s)-P(s')$.
\item $\forall (s,s')\in Q : F(s,s') > 0$, $y(s,s')- c(s,s') = P(s)-P(s') \geq 0$.
\item %$\forall s \in S$, 
$\sum_{s \in S} P(s) (\deg_{i}(s) - \deg_{o}(s)) = \sum_{(s,s')\in Q} F(s,s')(p(s,s')-c(s,s'))$, where  $\deg_i(s)$ and $\deg_o(s)$ are the number of drivers to enter and leave the platform at the state $s$ respectively.
\end{enumerate}
\end{lemma}

\vvspace{-1ex}
Leveraging the power of Lemma~\ref{lem:potential}, we are able to prove the following theorem 
%(proof details can be found in Appendix~{\ref{appendix:proof_feasib}}) 
stating that a fair reward re-allocation scheme always exists in all non-degenerating scenarios (i.e., the total revenue is non-negative and at least one driver starts from a non-terminal state).
\vvspace{-0.5ex}
\begin{theorem}
    \label{thm:feasiblility}
    %{\color{red} [Double check the conditions. Do we need to add $c_{out}$?]}
Let $S_\#\subseteq S$ denote the set of terminal states. 
%and let $R_* := \{(s,s')\in S\times S: F(s,s')>0\}$ denote the set of spatiotemporal arcs which at least one driver drives through. 
If there exist $s_1 \in S \setminus S_\#$ and $s_2\in S$ such that $F(s_1,s_2)>0$ and $\mathcal{I} \ge \sum_{(s,s')\in Q} F(s,s')\cdot c(s,s')$ (recall $\mathcal{I}$ is the total income collected from the riders), then there exists a fair reward allocation plan.
    \label{p2:feasible}
\end{theorem}

\begin{proof}
    %Denote $\P = \sum_{(s,s')\in R_*} p(s,s')$ and $\C = \sum_{(s,s')\in R_*} c(s,s')$. If $\P=\C$, we just let $r(s,s')=c(s,s')$ so that all drivers have zero revenue. Now we assume $\P>\C$, and we derive a possible way to allocate the total revenue.
    
    We define a directed graph $G'$ on vertex set $V(G')=(S-S_\#) \cup \{t\} $, in which all states in $S_\#$ are contracted in a single vertex $t$. For each order from $s\notin S_\#$ to $s'$ we add a directed edge $(s,s')$ with length 1 if $s'\notin S_\#$, or $(s,t)$ with length 1 if $s' \in S_\#$, and for each possible cruise arc from $s$ to $s'$ we add an edge with length 0. 
    
    As all arcs advance in time, the graph is a directed acyclic graph (DAG). Therefore, we can define $\tilde{P}(s)$ as the maximum distance of all paths from $s$ to $t$, or $0$ if $s\in S_\#$.
    Then we let $R_* = \{(s,s')\in Q: f(s,s')>0\}$, define $\mu(s,s') = \tilde{P}(s) - \tilde{P}(s')$, and then we allocate the revenue proportional to $\mu$, i.e. let
    
    \begin{equation}
        P(s) = \tilde{P}(s) \cdot \frac{\mathcal{I} - \sum_{(s,s')\in Q} F(s,s') c(s,s')}{\sum_{(s,s')\in Q} F(s,s')\mu(s,s')}.
    \end{equation}
    
    Because of the assumption that $\mathcal{I} \ge \sum_{(s,s')\in Q} F(s,s')c(s,s')$, we are ensured that $r(s,s')-c(s,s')$ is proportional to $\mu(s,s')$ with a non-negative ratio. We can see all constraints are satisfied.
\end{proof}

\vvspace{-1ex}

When the fair re-allocation scheme is not unique, we solve the quadratic program in Figure~\ref{fig:QP} to find the scheme to minimize the total squared distortion between the price paid by the rider and the reward allocated to the driver among all trips. In this way, we try the best to let the reward re-allocation reasonably reflects the real income generated by driving through each arc.
It is straightforward to see that the constraints (\ref{p21:con1},\ref{p21:con2},\ref{p21:con4},\ref{p21:con5},\ref{p21:con6}) in the quadratic program implement the conditions stated in Lemma~\ref{lem:potential}.

\begin{figure}[h]
\begin{center}
\fbox{\begin{minipage}[t]{0.95\columnwidth}%
\begin{small}
\vvspace{-2.5ex}
\begin{align}
    & \textup{Minimize~} \sum_{F(s,s')>0} F(s,s') (p(s,s')-y(s,s'))^2 \nonumber \\
    & \textup{Subject to~ }   P(s)\ge 0, ~~ \forall s\in S  \\
    & \qquad y(s,s')  = P(s)-P(s')+c(s,s'),~~\forall F(s,s')>0 \label{p21:con1}\\
    & \qquad y(s,s')  \le P(s)-P(s')+c(s,s'),~~\forall (s,s')\in Q \label{p21:con2}\\
     %P(s) \ge 0, ~~\forall s\in S \label{p21:con3}\\
    & \qquad  P(s) = 0, ~~\forall s\in S_\# \label{p21:con4}\\
    & \qquad  y(s,s') \ge c(s,s'),~~\forall F(s,s')>0 \label{p21:con5}\\
    &  \qquad  \sum_{s \in S} P(s) (\deg_{i}(s) - \deg_{o}(s)) \nonumber
\end{align}
\vvspace{-5ex}
\begin{align}
    & \qquad \qquad \qquad = \sum_{(s,s')\in Q} F(s,s')(p(s,s')-c(s,s')) \label{p21:con6}
\end{align}
\vvspace{-2.5ex}
\end{small}
\end{minipage}}
\end{center}
\vvspace{-3ex}
\caption{Quad.\ Prog.\ with decision variables $\{P(s)\}_{s \in S}$}
\label{fig:QP}
\end{figure}

\section{The Stochastic-Demand Setting}\label{sec:stochastic-setting} 
%\label{sec:learning}
\vvspace{-0.5ex}

In the previous sections, we studied the optimal car dispatching and reward allocation task assuming the access to the full list $R$ of latent orders, which is not realistic in practice. In this section, we assume that $R$ is drawn from an \emph{unknown} distribution $\{\calD(s,s')\}$ and address the problem with techniques combining both learning and optimization. To achieve this goal, we 
%first 
study the optimal car dispatching and reward allocation task with the distribution $\{\calD(s,s')\}$ \emph{known}. We will refer to this task as the \emph{stochastic-demand setting}. %Then, we adapt the Thompson sampling  framework to simultaneously learn the distribution $\{\calD(s,s')\}$ and optimize the scheduling plan.

%In this section, we introduce our algorithm to the maximum revenue car dispatching in the learning setup. First, we introduce the stochastic-demand setting in which latent orders follow certain distributions, and then introduce a learning scheme that utilizes Thompson Sampling to learn the demand distribution based on  observed historical data while optimizing the expected revenue.

%\subsection{The Stochastic-Demand Setting}

Our algorithm for the stochastic-demand setting is a natural extension of that for the deterministic setting presented in the previous sections. Below we describe the adaptation we make for each phase in the deterministic setting. We will also introduce a special parametric demand distribution (Gaussian-Poisson distribution) for the learning algorithm in Appendix~\ref{app:gaussian-poisson}.

\paragraph{Phase 1: Revenue Optimization.} For each arc $(s,s')$, we denote $x_\ss$ as the random variable for the number of latent orders, $\{v_t\}_{t\in[x_\ss]}$ as the random variables for the valuations, and denote $\calD(\ss)$ as the distribution of $(x_\ss,\{v_t\}_{t\in[x_\ss]})$, with the assumption that each $v_t$ are {\it i.i.d.}~variables.

%\xnote{This definition does not show the ``stochastic''. Need to highlight which part is stochastic. NLWC is the deterministic problem. }\comment{added.}

 %(``\emph{qualified} orders'') %with valuation at least $p$, %\xnote{I cannot follow why we restrict to ``with valuation at least $p$''? Moreover, in (8), it is $P(v \le p)$?} \comment{Because if valuation is less than $p$ then the rider will not take the trip. It should be $\ge$, a typo} 
 
For each arc $(\ss)$, if we fix the price to be $p$ and plan to dispatch $n$ drivers to the arc, the number of the fulfilled latent orders will be the smaller value of $n$ and the number of  orders of valuations at least $p$. Therefore, given $\calD(\ss)$, we may compute the following quantities:

\begin{itemize}
\item The probability mass function $\mathcal{P}(i; \ss, p) :\N\to\R$ for the number of qualified orders (orders with valuation at least $p$):
$\mathcal{P}(i; \ss, p) = \sum_{j=0}^\infty b(i,j;\Pr[v_t\ge p])\Pr[x_\ss = j]$,
where $b(k,n;P)=\binom{n}{k}P^k(1-P)^{n-k}$ computes the binomial distribution.
\vvspace{-0.5ex}
\item Let $\tilde{u}(n; \ss, p)$ be the number of the fulfilled latent orders; its expectation:
$
        \E[\tilde{u}(n; \ss, p)] = \sum_{i=0}^\infty \mathcal{P}(i; \ss, p)\min\{i,n\}$.
\vvspace{-0.5ex}
\item The expected revenue on $(s,s')$ at price $p$ and $n$ drivers:
$        \mathcal{R}(n,p;\ss) = p\cdot \E[\tilde{u}(n; \ss, p)] - c(\ss)\cdot n$.
\end{itemize}

The following definition states the optimization problem we have to solve in order to maximize the revenue in car dispatching in the stochastic-demand setting.

\begin{definition}\label{def:stochastic-mrcd}
Given $\calD(\ss)$ for all arcs $(\ss)$, the \emph{Stochastic Maximum Revenue Car Dispatching} problem is to find the optimal solution to the NLWC problem on the directed graph $(V_0, E_0)$, where $(V_0, E_0)$ is constructed in a similar way as described above Proposition~\ref{prop:mrcd}, and the only difference is that for the edges corresponding to driving with a rider, we set the corresponding reward function $r(n; e^{(\mathrm{w})}_{s,s'}) =\max_{p\in \R^{\ge 0}} \{\mathcal{R}(n,p;\ss)\}$.
\end{definition}
\vvspace{-1ex}

In Definition~\ref{def:stochastic-mrcd},  $r(n,e_{s,s'}^{(w)})$ is re-defined so as to equal the maximum possible (over all candidate prices) expected revenue generated by dispatching $n$ drivers to the arc $(s,s')$. Therefore, the optimal solution to the stochastic maximum revenue car dispatching problem is the maximum possible expected revenue achieved by any dispatching plan.

Note that in Definition~\ref{def:stochastic-mrcd}, the only quantity that specifically depends on the form of the demand distribution is the non-linear reward function on the edges $e_\ss^{\mathrm{(w)}}$. 

\vvspace{-1ex}

\paragraph{Phase 2: Fair Re-allocation.} %This phase is similar to that of the deterministic setting, and therefore we defer the procedure and its fairness guarantees to Appendix~\ref{app:stochastic-reward-reallocation}.

After solving the NLWC problem on $(V_0, E_0)$, we obtain the number of drivers to dispatch and the price for each arc $(s, s')$. With this information, we may invoke the quadratic program in Figure~\ref{fig:QP} to find out the potential-based reward re-allocation scheme for the drivers. We are able to show the following the fairness guarantees in the stochastic-demand setting, while the detailed proof is omitted since it is almost the same as the proof in Phase 2 of the deterministic setting.

%When paired with Phase 2, the scheme satisfies a similar fairness constraint to Definition \ref{def:fair}, only with the modification that the budget-balance is in the sense of \emph{expectation}.
%\footnote{Eqn. (\ref{eqn:stoc:rev:function}) is monotonic increasing for $x\in \mathbb{N}$, but the increment becomes small when $x>\tilde{\lambda}+O(\sqrt{\tilde{\lambda}})$, so we can truncate the reward function and only introduce $\lambda+O(\sqrt{\lambda})$ linear edges for an edge in $E_1$.} 

\begin{theorem}\label{thm:stoc:fairness}
In the stochastic-demand setting, the potential-based reward re-allocation scheme obtained by the QP in Figure~\ref{fig:QP} satisfies the fairness conditions in Definition~\ref{def:fair}, except for that the budget-balance condition is changed to the following expectation version.
\begin{itemize}
\item \underline{\emph{Expected}-budget-balance.} The expected income collected from the riders should equal to the amount to be distributed to the drivers.\footnote{Similarly, here we also omit the amount that the platform would like to keep for profit.} Formally, it is required that
$\sum_{(s, s') \in Q} y(s, s') \cdot F(s, s') = \E[\mathcal{I}]$, where $\mathcal{I}$ is the collected income.
\end{itemize}
\end{theorem}

\vvspace{-1ex}
\paragraph{Online Learning.} We use a Thompson sampling-based algorithm to learn the demand distributions from riders' responses to given prices. The details are deferred to Appendix~\ref{sec:online-learning}. 

\vvspace{-1.5ex}
%Therefore, we assume that latent orders on arc $(\ss)$%, denoted as $x_\ss$, obeys the Poisson distribution $Poi(\lambda_\ss)$, and the valuation of every rider on arc $(\ss)$ obeys the normal distribution $N(\mu_\ss,\sigma_\ss^2)$. 
%follow the Gaussian-Poisson distribution with parameters $(\mu_\ss,\sigma_\ss,\lambda_\ss)$.
%For convenience, for an arc $(s,s')$ we denote $(\mu_\ss,\sigma_\ss,\lambda_\ss)$ as the \emph{distribution parameters} of arc $(\ss)$. 

\section{Experimental Evaluation} \label{sec:experiments}
\vvspace{-1ex}
Due to space constraints, we defer many of the experiments to Appendix~\ref{app:add-exp}. For example, we empirically verify the regularity of Gaussian-Poisson distributions in Appendix~\ref{app:add-exp-regularity}, and evaluate the online learning algorithm in Appendix~\ref{app:add-exp-robustness}; we also show an illustrative example of our fair re-allocation algorithm on the real-world dataset in Appendix~\ref{app:add-exp-example}. 

Experiments are run on an Intel i7-8750H, 24GB RAM computer with MATLAB 2021b. % I saw on CMT that the device data should be included...
\vvspace{-1ex}
\subsection{Model Setting}
\vvspace{-0.5ex}
We now evaluate our algorithm by simulated experiments on the DiDi Chuxing public dataset \cite{Didi2021} collected from the real-world ridesharing in Chengdu, China.
For one day, we extract all orders and driver initial positions and discretize the locations into $10\times 10 = 100$ squares with dimension $2\mathrm{km}\times 2\mathrm{km}$, and divide the time interval between 8am and 1pm in a day into 20 slots, each of which spans 15 minutes. Therefore, there are 2000 spatio-temporal states in a day, and we use the reward column in the dataset as the rider's valuation for the trip. Finally, we assume the latent orders follow the Gaussian-Poisson distribution (Appendix~\ref{app:gaussian-poisson}), and collect the data for 30 days and fit the numbers and valuations of orders in any arc into the Gaussian-Poisson distribution, as the true model parameters. 

\textbf{Robustness.}
To evaluate the generalization ability of our algorithm, we modify the following two key parameters in experiments: the number of drivers and the standard deviations of the riders' valuations. Here we report the experimental results showing that our algorithms still perform well under these different experimental environments.

In Table~\ref{table:drivers1}, we modify the number of drivers. In the $50\%$ drivers setting we remove each driver with $50\%$ independent probability and in the $200\%$ drivers setting we duplicate every driver. In Table~\ref{table:sigma1}, we modify the standard deviations of the riders' valuations
%. Compared to the original dataset, we modify the standard deviations of valuations 
by $0.5$ and $1.5$ times respectively.

\vvspace{-1ex}

\subsection{Revenue Evaluation}

 Given the true model parameters, we invoke the algorithms described in Section~\ref{sec:stochastic-setting} to find out the offline (model parameters known) optimal revenue of the Stochastic Maximum Revenue Car Dispatching problem. We refer to this value as the \emph{two-phase value} ({\sf 2P}). For comparison, we introduce the baseline \emph{distance-based fix-price algorithm} ({\sf FP}) where the price for each arc is proportional to the distance of the trip with a globally fixed (but tuned) ratio, and the dispatching is done via the same network-flow-based planning algorithm. 

\vvspace{-1ex}
%\paragraph{Results.} We report that the {\sf OV} revenue is $9.32 \times 10^4$ and ${\sf FP}$ revenue is $7.56 \times 10^4$. Our optimal algorithm collects $23\%$ more revenue than ${\sf FP}$, which is commonly adopted in current the taxi system.

\subsection{Fairness Evaluation}

To evaluate the fairness, we define the $A(s)$ as the \emph{average net income} of all drivers initially at state $s$. For a driver $q\in Q$, we denote $s_q$ as the initial state of $q$ and $u_q$ as the total net income of $q$. Then, we define the absolute unfairness
    $\Xi = \sqrt{\frac{\sum_{q\in Q} (u_q-A(s_q))^2}{|Q|}}$
and relative unfairness $\xi = \Xi / \frac{\sum_{q\in Q} u_q}{|Q|}$, which can be interpreted as the absolute and relative fluctuation of drivers' net incomes from given initial states. We have proven that our two-phased algorithm guarantees \emph{zero} unfairness, and evaluate the unfairness of baseline pricing algorithms. To show the contribution of re-allocation, we refer to the result of only Phase 1 as {\sf P1}.
\vvspace{-1ex}
\subsection{Results}

We report the revenue ({\sf Rev}) and relative unfairness ({\sf Unf}) of different settings in following tables.
\vvspace{-1ex}
%\begin{small}
\begin{table}[hbt]
\centering
\begin{tabular}{|c|c|c|c|c|c|c|c|c}
\hline
   \#drivers & \multicolumn{2}{c|}{6655} & \multicolumn{2}{c|}{13411} & \multicolumn{2}{c|}{26822} \\ \hline
   & {\sf Rev} & {\sf Unf} & {\sf Rev} & {\sf Unf} & {\sf Rev} & {\sf Unf} \\ \hline
{\sf 2P} & \bf 6.82 &\bf 0.000  &\bf 9.32 &\bf  0.000  &\bf 11.17 &\bf 0.000 \\ \hline
{\sf P1} &\bf  6.82 &  0.114  &\bf 9.32 &  0.172  &\bf 11.17 & 0.243 \\ \hline
{\sf FP} &  5.54  & 0.108 &  7.56  & 0.168  & 9.02 & 0.244 \\ \hline
%{\sf EE}$_2$ &  &   & \\ \hline
\end{tabular}
\vvspace{-1ex}
\caption{\small{{\sf Rev}$(\times 10^4)$/{\sf Unf} with different numbers of drivers }.}
\label{table:drivers1}
%\medskip
\centering
\begin{tabular}{|c|c|c|c|c|c|c|}
\hline
   stddev    & \multicolumn{2}{c|}{$0.5\sigma$} &  \multicolumn{2}{c|}{$1.0\sigma$} &  \multicolumn{2}{c|}{$1.5\sigma$} \\ \hline
   & {\sf Rev} & {\sf Unf} & {\sf Rev} & {\sf Unf} & {\sf Rev} & {\sf Unf} \\ \hline
{\sf 2P}  &\bf 10.36 &\bf 0.000  &\bf 9.32 &\bf 0.000   &\bf 8.61 &\bf 0.000 \\ \hline
{\sf P1} &\bf  10.36 & 0.167 &\bf 9.32 & 0.172  &\bf  8.61 & 0.178   \\ \hline
{\sf FP} & 7.90 & 0.162  &  7.56  & 0.168 & 7.25 & 0.172 \\ \hline
%{\sf EE}$_2$ &  &   & \\ \hline
\end{tabular}
\vvspace{-1ex}
\caption{\small{{\sf Rev}$(\times 10^4)$/{\sf Unf} with modified standard deviations }.}
\label{table:sigma1}
\end{table}
%\end{small}
\vvspace{-2.5ex}

We see that our algorithm achieves higher revenue than the fixed-price algorithm, and our re-allocation phase eliminates the unfairness that would typically range from $10\%$ to $25\%$ of drivers' incomes, which increases with numbers of drivers. 

Intuitively, a large number of drivers would tend to result in increased unfairness as they fulfill a large portion of latent orders with a wider spread of profits (analysis in Appendix~\ref{app:unfairness:drivers}). Therefore, the re-allocation phase becomes essential for satisfaction of drivers especially in this scenario.

\vvspace{-2ex}
\section{Computational Complexity Analysis} \label{app:complexity}

Let $n,m,a$ be the number of states, latent orders and admissible arcs, respectively. Our Phase 1 essentially solves a linear program of size $O(m+a)$, which runs in $\tilde{O}((m+a)^{2.373})$ time \cite{LPComplexity}. Our Phase 2 solves a quadratic problem of $O(n)$ variables and $O(a)$ input size, which can be transformed into a semidefinite program that runs in $\tilde{O}(\sqrt{n}(an^2+a^{2.373}+n^{2.373}))$ time \cite{jiang2020faster}.

\section{Conclusion}
\vvspace{-0.5ex}

In this paper, we present an algorithmic framework for car dispatching and pricing with both revenue and fairness guarantees. Empirical evaluation shows that our method performs better than the baseline alternatives in the real-world dataset. For future directions, it is interesting to prove the regularity of edge demand functions in Gaussian-Poisson distribution and explore the regularity property of other distributions, and mathematically prove the guarantees of our Thompson Sampling algorithm (e.g., its convergence property and finite-sample regret bound).

\section*{Acknowledgments}

Zishuo Zhao would appreciate Shiyuan Wang for discussions in relevant topics and Shuran Zheng for valuable comments in improvement of a preliminary version. Zishuo Zhao is supported in part by NSF CCF-2006526.

%\newpage

%\bibliographystyle{aaai}
%\comment{Reference: use full names(ICML), journal need page number, conference inproceeding}
\clearpage

\bibliographystyle{named}

\begin{small}
\bibliography{ref}
\end{small}
%\bibliographystyle{aaai-named}
%\bibliography{ref}

\clearpage

\appendix

\setcounter{section}{0}
\renewcommand{\thesection}{\Alph{section}}

\section{Omitted Proofs}

%\subsection{Proof of Remark \ref{thm:NLWC:validity}}
%\begin{proof}
    %We notice that the reward function on each edge exactly shows the total revenue gained by drivers as a function of number of drivers taking orders and cruising on that arc. Therefore, we just need to dispatch as many drivers on an arc as the optimal flow on it.
%\end{proof}

\subsection{Proof of Theorem \ref{thm:nphard}}
\label{appendix:proof:nphard}

\begin{proof}
    In this proof, we set all costs to be zero, so the reward is equivalent to the revenue.
    
    We notice that even for states that are not adjacent, the reward function $r(n;e_\ss^{(\mathrm{w})})$ is still well-defined if for all states $t$ on a path from $s$ to $s'$, all drivers visiting $t$ must visit $s$ before and visit $s'$ after. In this concept, we can regard it as a ``virtual arc'' as long as they do not intervene with each other. Then, we reduce Set Cover 
    \cite{Karp1972} 
    to Maximum Revenue Car Dispatching.
    
    \begin{lemma}
        For $n, A\in \mathbb{N}, n\ge 1$, we can construct an arc $(s,s')$ in polynomial time and size with $r(x;e_\ss^{(\mathrm{w})}) = A\cdot 1_{x\ge n} + C_1$ for $x\in [n]$, in which $C_1$ is a constant only dependent on $A$ and $n$. \label{lemma:cons1}
    \end{lemma}
    
    \begin{proof}[Proof of Lemma \ref{lemma:cons1}]
        We create an arc $(s,s')$ with $n$ orders of valuation $\frac{n-1}{1}A,\frac{n-1}{2}A,\cdots,\frac{n-1}{n-1}A,A$. Then it satisfies the condition with $C_1 = A\cdot (n-1)$.
    \end{proof}
    
    \begin{lemma}
        For $n, A\in \mathbb{N}, n\ge 1$, we can construct a virtual arc $(s,s')$ in polynomial time and size with $r(x+(n-1);e_\ss^{(\mathrm{w})}) = A\cdot \max\{0, x-1\} + C_2$ for $0\le x \le n$, in which $C_2$ is a constant only dependent on $A$ and $n$. We call it a $(A,n)$-virtual arc. \label{lemma:cons2}
    \end{lemma}
    
    \begin{proof}[Proof of Lemma \ref{lemma:cons2}]
        For each $i\in [n-1]$, we construct an arc $(s,s_i)$ with $r(x+(n-1);e_{s,s_i}^{(\mathrm{w})}) = A(i\cdot 1_{x_i\ge i+2} + C_1^{(i)})$ for $x_i>0$ by Lemma \ref{lemma:cons1}, and an arc $(s_i,s')$ with no reward. Then, as it is straightforward to see $i\cdot 1_{x_i\ge i+1} \le x_i - 1 - 1_{x_i\ge 2}$ for $x_i>0$ and $C_1^{(i)}-1 =i(i+1) - 1>0$ for $x_i = 0$, we always have
        
        \begin{equation*}
        r\left(x_i;e_{s,s_i}^{(\mathrm{w})}\right) \le A (x_i - 1 - 1_{x_i\ge 2} + C_1^{(i)}).
        \end{equation*}
        
        Therefore,
        \begin{align*}
            &~~r\left(\sum_{i=1}^{n-1} x_i; e_\ss^{(\mathrm{w})}\right) = \sum_{i=1}^{n-1} r(x_i;e_{s,s_i}^{(\mathrm{w})}) \\
            %&= A\sum_{i=1}^{n-1} (i\cdot 1_{x_i \ge i+2} + C_1^{(i)}) \\
            &\le A\sum_{i=1}^{n-1} (x_i - 1 - 1_{x_i\ge 2} + C_1^{(i)}) \\
            &= A\left(\sum_{i=1}^{n-1} x_i - (n-1) - \sum_{i=1}^{n-1}1_{x_i\ge 2} + \sum_{i=1}^{n-1}C_1^{(i)}\right).
        \end{align*} 
        We let $x = \sum_{i=1}^{n-1} x_i - (n-1)$, $C_2 = A\sum_{i=1}^{n-1}C_1^{(i)}$, and notice that $\sum_{i=1}^{n-1}1_{x_i\ge 2} \ge 1_{x\ge 1}$. Then we get:
        \begin{align*}
            &~~r(x+(n-1);e_\ss^{(\mathrm{w})}) \\
            &\le A(x - 1_{x\ge 1}) + C_2 \\
                        &= A\cdot \max\{0, x-1\} + C_2, 0\le x\le n.
        \end{align*}
        When we let $x_i = 1 + x \cdot 1_{i=x-1}$, the equality holds.
    \end{proof}
    
    Now consider an instance of the set cover problem with the set $A=\{a_1,\cdots,a_n\}$, a family $\mathcal{K}=\{K_1,\cdots,K_m\}$ of subsets of $A$. Now we construct the Maximum Revenue Car Dispatching problem with $S = A \cup \mathcal{K} \cup \{O\}$. In the initiation, on each $a_i\in A$ we assign 1 driver, and for each $K_j\in \mathcal{K}$, we assign $|K_j|-1$ drivers. Then, for each $a_i \in K_j$, we add an edge $(a_i,K_j)$ with one order of valuation 1, and for each $K_j$, we add an $(1,|K_j|)$-virtual arc $(K_j,O)$ as in Lemma \ref{lemma:cons2}, with the respective $C_2$ denoted as $C_{(j)}$. Then, the drivers initially in $K_j$ will go straight to $O$ getting $C_{(j)}$ reward, for a total of $C_3 := \sum_{j=1}^m C_{(j)}$.
    
    Now we consider the routes of drivers initially in $A$. Each order from $A$ to $\mathcal{K}$ earns 1, and the reward from any $K_j$ to $O$ is non-decreasing, so in an optimal plan all drivers must reach $O$, and all nodes in $\mathcal{K}$ visited by some drivers from $A$ form a set cover of $A$. For each fixed $K_j$, if $x\ge 1$ drivers from $A$ visit $K_j$, the virtual arc $(K_j, O)$ will earn an additional reward of $x-1$. Therefore, if totally $k$ nodes in $\mathcal{K}$ are visited, the total revenue is:
    \[
    2n + C_3 -k.
    \]
    Therefore, if we can compute the optimal plan for this instance of Maximum Revenue Car Dispatching, we find the optimal solution to Set Cover, so there is no polynomial time algorithm for general Maximum Revenue Car Dispatching unless $\mathrm{P}=\mathrm{NP}$.
\end{proof}

%\subsection{Proof of Theorem \ref{thm:reg_equiv}}
%\label{appendix:proof_reg_equiv}

\subsection{Proof of Lemma \ref{lem:potential}}
\label{appendix:proof_lem_potential}

\begin{proof}

    We firstly prove the ``only if'' direction, i.e. a reward re-allocation is fair only if there exists a potential function satisfying the condition (1-4):
    
    We assume that the plan $\mathcal{P}$ is fair. Because $\mathcal{P}$ is envy-free, for every driver who visits the same state $s$, their utility obtained from $s$ to the end must be the same, and we denote it as $P(s)$. We only consider states visited by at least one driver. For all $s$ not visited by any driver, we assign $P(s)=0$.
    
    In condition 3, let $d$ be a driver who drives through the arc $(\ss)$. By construction rule of $P$, the utility of $d$ from $s$ to the end is $P(s)$ and the utility of $d$ from $s'$ to the end is $P(s')$, so the net income of $d$ driving from $s$ to $s'$ must be $P(s)-P(s')$. According to the individual rationality requirement in fair re-allocation, $P(s)-P(s')\ge 0$. So condition 3 holds.
    
    If condition 2 is violated, then from condition 3 we know that $F(s,s')=0$. Then, for any driver at $s$, when he/she deviates the route and drives to $s'$ instead, he/she will benefit from the deviation. Contradiction. So condition 2 holds.
    
    In condition 1, we consider a driver $d$ leaving the platform at $s$, then $d$ will not earn any net income further. By construction rule of $P$ we know $P(s)=0$.
    
    In condition 4, the LHS is the summation of total net income of all drivers and the RHS is the summation of revenue of the platform, so it is equivalent to budget-balance condition. Therefore condition 4 holds.
    
    In conclusion, we can construct a potential function $P$ from any fair reward re-allocation.

\medskip 

    Then we prove the ``if'' direction, i.e. a reward re-allocation is fair if there exists a potential function satisfying the condition (1-4):
    
    If a re-allocation scheme is not fair, then at least one of budget-balance, individual rationality, subgame-perfectness and envy-freeness is violated. We assume there still exists a potential function $P$ satisfying all conditions.
    
    If envy-freeness is violated, then there must exist a state $s$ such that two drivers earn different net incomes from $s$ to the end. By condition 3, all drivers at state $s$ earn a net income of $P(s)$ from $s$ until he/she leaves the platform if he/she follows the dispatching plan. Contradiction.
    
    If subgame-perfectness is violated, then there must be a state $s$ where a driver may deviate and improve his/her utility. Then condition 2 is violated. Contradiction.
    
    If budget-balance is violated, then condition 4 is violated. Contradiction.
    
    If individual-rationality is violated, then the ``$\ge 0$'' constraint in condition 3 is violated. Contradiction.
    
    In conclusion, if a re-allocation is not fair, such $P$ satisfying all conditions does not exist.
\end{proof}

%\subsection{Proof of Theorem \ref{thm:feasiblility}}
%\label{appendix:proof_feasib}

\section{Approximate NLWC Algorithm for Non-regular Cases}
\label{appendix:algo:non_regular}

To make the edge reward function concave, we essentially construct the \emph{concave envelope} of $r(\cdot;e_\ss^{\mathrm(w)})$, as the least-valued concave function $\overline{r}(\cdot;e_\ss^{\mathrm(w)})$ not less than it. To compute $\overline{r}(\cdot;e_\ss^{\mathrm(w)})$, we only need to solve a linear program with decision variables $\{\overline{r}(i;e_\ss^{\mathrm(w)})\}_{i\in [o(\ss)]}$:

\begin{center}
\fbox{\begin{minipage}[t]{0.95\columnwidth}%
\begin{small}
\begin{align}
    & \textup{Minimize~} \sum_{i=1}^{o(\ss)}  \overline{r}(i;e_\ss^{\mathrm(w)}) \nonumber \\
    & \textup{Subject to~ }   \overline{r}(i;e_\ss^{\mathrm(w)})\ge r(i;e_\ss^{\mathrm(w)}), ~~ \forall i\in [o(s,s')]  \\
    & \qquad 2\overline{r}(i;e_\ss^{\mathrm(w)}) \ge \overline{r}(i-1;e_\ss^{\mathrm(w)}) + \overline{r}(i+1;e_\ss^{\mathrm(w)}), \\
    & \qquad \qquad \qquad \qquad \qquad~~ i = 2,\cdots,o(s,s')-1. \nonumber
\end{align}
\end{small}
\end{minipage}}
\end{center}

Then, in Line 7 of Algorithm A1, we compute $w$ with $\overline{r}$ instead, thus making $w$ non-increasing and getting an approximated regular NLWC instance for computation of approximated Maximum Revenue Car Dispatching.

\section{An Example for Merit of Two-Phase Pricing}
\label{appendix:example_two_phase}

Through Wuhan City runs the Yangtze River, across which there had been only two bridges in 2000s. During rush hour, people traveling across the river crowded the bridges and made the traffic extremely heavy. As taxis would struggle in crossing the river, which would take a long time and increase both fuel and opportunity costs, taxi drivers frequently refused trips and made citizens complain.\footnote{See https://www.wsj.com/articles/SB10001424052702303330

204579247731532836694}

Consider three locations $W_1,W_2,H$, in which $W_1,W_2$ are far apart but both in Wuchang on the same side of the river, while $H$ is just opposite to $W_1$ across the No.2 Yangtze Bridge in Hankou. By the taxi pricing system, the fare from $W_1$ to $W_2$ is $p(W_1,W_2)=20$, while $p(W_1,H)=10$. However, due to distinct traffic conditions, the costs are $c(W_1,W_2)=10, c(W_1,H)=8$. As $p(W_1,W_2)-c(W_1,W_2)>p(W_1,H)-c(W_1,H)$, taxi drivers would not be willing to cross the river.

Suppose there were two riders who would travel from $W_1$ to $W_2$ and $H$, respectively, with valuations the same as taxi fares. In the conventional mechanism, two drivers would take one order each, but one earns $10$ while the other earns only $2$, making the latter complain or even refuse the trip. If we enforce envy-freeness, as we cannot increase the $p(W_1,H)$ (otherwise the price exceeds the rider's valuation and the rider would not take the trip), we can only lower $p(W_1,W_2)$ to $12$, which just lowers the revenue and makes drivers \textbf{equally unsatisfied}. 

In our two-phase pricing mechanism, we can compute the potentials $P(W_1)=6,P(W_2)=P(H)=0$. Therefore, $r(W_1,W_2)=16,r(W_1,H)=14$, so no matter which trip they choose, they always earn $6$. This kind of redistribution has not been possible until ridesharing platforms occur, but does make drivers envy-free and alleviate the problem that drivers are not willing to cross the river (and other situations of heavy traffic) without modifying rider-side pricing or total revenue, improving both parties' experience.

%\section{Reward Re-allocation in the Stochastic Demand Setting and its Fairness Guarantees} \label{app:stochastic-reward-reallocation}

\section{Computing the Reward Functions for the Gaussian-Poisson Demand Distribution} 

\subsection{The Gaussian-Poisson Demand Distribution.} \label{app:gaussian-poisson}

In literature on pricing under stochastic demands, it is common to model the stochastic arrivals as Poisson processes~\cite{jiang2005estimating,poisson1,poisson2,poisson3}, and assume that the agents' undisclosed valuations of follow the normal distribution~\cite{jiang2005estimating,kashyap2018auction}. In light of this, we define the parametric \emph{Gaussian-Poisson distribution} for $\calD(\ss)$, which is both practically useful and easy to learn. If $\calD(\ss)$ is the Gaussian-Poisson distribution with parameters $(\mu_{\ss},\sigma_{\ss},\lambda_{\ss})$, we have that $x_\ss\sim \Pois(\lambda_{\ss})$ and each $v_t\sim \mathcal{N}(\mu_{\ss},\sigma^2_{\ss})$, where $\Pois(\lambda)$ and $\mathcal{N}(\mu,\sigma^2)$ respectively denote the Poisson and the Gaussian distribution.

For parametrization of the demand distribution, in our experiments in stochastic setting and online learning setting, we assume the latent orders obey Gaussian-Poisson distribution.

\subsection{Computation of Reward Functions} \label{app:compute-reward-function-gaussian-poisson}

Fix any arc $(\ss)$ and the distribution parameters $(\mu_\ss,\sigma_\ss,\lambda_\ss)$. Denote $\Phi(\cdot)$ as the cumulative distribution function (cdf) of standard Gaussian distribution. If we offer a price $p$, since valuations of the latent orders on the arc obey $\mathcal{N}(\mu_\ss,\sigma_\ss^2)$, each latent order has a valuation greater than $p$ independently with probability
$(1-\Phi(\frac{p-\mu_\ss}{\sigma_\ss}))$. For convenience, we also refer to these orders as \emph{qualified}.  

The following lemma characterizes the number of the qualified orders on the given arc. %The proof of Lemma \ref{lem:distribution} can be found in Appendix \ref{appendix:proof_distrib}.

\begin{lemma}\label{lem:distribution}
Let $\tilde{x}(p)$ denote the the number of the qualified orders on the arc $(\ss)$. $\tilde{x}(p)$ follows $\Pois(\tilde{\lambda}(p; \ss))$ where
\[
\tilde{\lambda}(p;\ss) := \left(1-\Phi(\frac{p -\mu_\ss}{\sigma_\ss})\right) \lambda_{\ss}.
\]
\end{lemma}

\begin{proof}
We equivalently prove the following statement: for $x\sim \Pois(\lambda)$ we toss $x$ coins each with head probability $p$, then the number of heads of all coins tossed obeys distribution $\Pois(\lambda p)$.

    The probability generating function of $\Pois(\lambda)$ is
    \begin{equation}
        \begin{aligned}
        G_1(t) &= \sum_{i=0}^{+\infty} e^{-\lambda}\frac{\lambda^i}{i!} \\
        &= e^{\lambda(t-1)}.
        \end{aligned}
    \end{equation}
    
    For every coin the probability generating function of the number of heads is
    \begin{equation}
        G_2(t) = (1-p) + pt.
    \end{equation}
    
    Then the probability generating function of the total number of tails is
    \begin{equation}
        \begin{aligned}
            G_1(G_2(t)) &= e^{\lambda((1-p+pt)-1)} \\
            &= e^{\lambda p(t-1)},
        \end{aligned}
    \end{equation}
    identical to the probability generating function of $\Pois(\lambda p)$.
    
    Therefore, the number of heads obeys the distribution $\Pois(\lambda p)$.
\end{proof}

For $\tilde{\lambda} = \tilde{\lambda}(p;\ss)$, we define the function 
\[
\Theta(n,\tilde{\lambda}) := n - \sum_{i=0}^{n-1} (n-i)\frac{\tilde{\lambda}^i}{i!}e^{-\tilde{\lambda}}.
\]
We have that $\Theta(n,\tilde{\lambda})$ is the expected number of the fulfilled orders on the arc if we dispatch $n$ drivers. This is because the expected number of fulfilled orders is
\begin{align*}
 &\sum_{i=0}^{\infty} \min\{n,i\} \cdot \Pr[\tilde{x}(p)=i] \\
&= \sum_{i=0}^{\infty} (n \cdot \Pr[\tilde{x}(p)=i]) - \sum_{i=0}^{n-1} ((n - i) \cdot \Pr[\tilde{x}(p)=i]) \\
&= n - \sum_{i=0}^{n-1} (n-i)\frac{\tilde{\lambda}^i}{i!}e^{-\tilde{\lambda}} = \Theta(n,\tilde{\lambda}).
    \end{align*} 

Finally, we use the functions defined above to compute $\mathcal{R}(n,p;\ss)$, and derive the calculation methods for the edge reward function $r(\cdot;e_\ss^{\left(\mathrm{w}\right)})$ as follows.

\begin{theorem}\label{thm:reduction}
If $\calD(s,s')$ follows the Gaussian-Poisson distribution with parameters $(\mu_\ss,\sigma_\ss,\lambda_\ss)$, then
\begin{align}
r(i;e_\ss^{\left(\mathrm{w}\right)})
 =\max_{p \in \mathbb{R}^{\ge 0}}\left\{ \Theta\left(i,\tilde{\lambda}(p;\ss)\right)p\right\}-c{(\ss)}i. \label{eqn:stoc:rev:function}
\end{align}
\end{theorem}

\section{Our Online Learning Algorithm}\label{sec:online-learning}

When the distributions $\{\calD(\ss)\}$ of the latent orders are not known beforehand, our scheduling algorithm needs to actively collect data and learn these distributions with better accuracy while pursuing higher revenue. Note that a key component of $\calD(\ss)$ is the riders' valuation distribution on each arc. The scheduling algorithm has to learn the distribution from the partial information that whether a rider has accepted the proposed price on the arc. On the other hand, the amount of partial information revealed about the valuation distribution critically depends on the scheduling algorithm's pricing strategy, as a too high or too low price would result in the riders always accepting or rejecting the offer, which is little useful information. Therefore, the learning-and-optimization algorithm has to carefully price the arcs to balance the two goals of obtaining enough information and securing high revenue. This is also known as the \emph{exploration vs.~exploitation} dilemma in online learning and decision-making.

\vvspace{-0.5ex}
\paragraph{The Thompson Sampling Framework.} To address this challenge, we adopt Thompson sampling (TS), a general online learning and decision-making algorithmic design principle that dates back to \cite{thompson1933likelihood} and proves to be useful in many practical tasks (e.g., \cite{chapelle2011empirical,kawale2015efficient,schwartz2017customer}). Suppose that the scheduling algorithm will run for a time horizon of $\mathcal{T}$ days, and each day forms an independent scheduling task with the identical distributions $\{\calD(\ss)\}$. The scheduling algorithm on day $\tau$ may use the information observed during the first $(\tau-1)$ days to learn $\{\calD(\ss)\}$, and has to make scheduling decisions for day $\tau$, generating revenue as well as new data for future learning. The TS framework usually works with parametric distributions (where we assumed that $\{\calD(\ss)\}$ are Gaussian-Poisson distributions with parameters $\{(\mu_{\ss}, \sigma_{\ss}, \lambda_{\ss})\}$), and maintain a prior distribution for the parameters. On each day $\tau$, TS samples the parameters $\{(\hat{\mu}^{(\tau)}_{\ss}, \hat{\sigma}^{(\tau)}_{\ss})\}$ ($\hat{\lambda}^{(\tau)}_{\ss}$ can be obtained from direct estimation as it is not involved in the exploration-exploitation dilemma) from the prior and correspondingly constructs the estimation $\{\hat{\mathcal{D}}^{(\tau)}(\ss)\}$. An optimal scheduling policy is computed based on $\{\hat{\mathcal{D}}^{(\tau)}(\ss)\}$ and the riders' responses (accept or reject) are observed. The TS algorithm then computes the posterior distribution for the parameters based on the new observation, which also serves as the prior on the next day.

\vvspace{-0.5ex}
\paragraph{Gaussian Priors and Laplace Approximation.} A key choice we have to make in designing the TS algorithm is the specific form of the prior distributions that should simultaneously guarantee the learning performance and facilitate the posterior calculation. In our algorithm, we set the prior distributions for both $\mu_{\ss}$ and $\sigma_{\ss}$ to be independent Gaussian distributions:
\vvspace{-0.5ex}
\[
\mu_{\ss} \sim \mathcal{N}(\mu^\mu_\ss, (\sigma^\mu_\ss)^2),~~ \sigma_{\ss} \sim \mathcal{N}(\mu^\sigma_\ss, (\sigma^\sigma_\ss)^2)
,
\vvspace{-0.5ex}
\]
where $\mu^\mu_\ss$, $\sigma^\mu_\ss$, $\mu^\sigma_\ss$, $\sigma^\sigma_\ss$ can be estimated based on the intrinsic properties of the trip $(s, s')$ (e.g., length, tolls, road quality, etc) without any interaction with the riders.

However, even with the above assumption, the posterior distributions of $\mu_{\ss}$ and $\sigma_{\ss}$ may become a complicated form other than Gaussian, which may lead to further description and computational complexity as the algorithm runs after multiple days. To address this challenge, we adopt the Laplace's method to approximate the potentially complicated posterior by another Gaussian distribution. Such a Laplace approximation method, first proposed by \citet{chapelle2011empirical}, is able to maintain the conjugacy properties for the priors and therefore greatly facilitates the computation. The detailed approximation procedure is derived in Appendix~\ref{app:laplace-approximation}.

\paragraph{Algorithm Description.} In Algorithm~\ref{alg:TS}, we describe the details of our TS algorithm. At Line~\ref{line:TS-5}, the $\lambda_{\ss}$ parameter is not involved in the exploration-exploitation dilemma and therefore is learned directly via the maximum likelihood estimate. At Line~\ref{line:TS-8}, the approximate Gaussian posterior is done via the Laplace's method.

\begin{algorithm}[h]
\caption{TS for Maximum Revenue Car Dispatching}
\label{alg:TS}
\begin{algorithmic}[1] 
\State {\bf for} each $(s, s') \in Q$: initialize \begin{small}
\[(\mu^{\mu,(1)}_\ss, \sigma^{\mu,(1)}_\ss, \mu^{\sigma,(1)}_\ss, \sigma^{\sigma,(1)}_\ss) \leftarrow 
(\mu^\mu_\ss, \sigma^\mu_\ss, \mu^\sigma_\ss, \sigma^\sigma_\ss).
\]
\end{small}
\For {$\tau \leftarrow 1, 2, \dots, \mathcal T$}
    \For{$(s, s') \in Q$}
        \State Sample $\hat{\mu}_{\ss}^{(\tau)} \sim \mathcal{N}(\mu^{\mu,(\tau)}_\ss, (\sigma^{\mu,(\tau)}_\ss)^2),~~ \hat{\sigma}_{\ss}^{(\tau)} \sim \mathcal{N}(\mu^{\sigma,(\tau)}_\ss, (\sigma^{\sigma,(\tau)}_\ss)^2)$.
        \State Estimate $\hat{\lambda}_{\ss}^{(\tau)}$ as the daily average of the number of the latent orders on $(s,s')$. \label{line:TS-5}
    \EndFor
    %\State{\bf end for}
    \State Compute the optimal Stochastic Maximum Revenue Car Dispatching (Definition~\ref{def:stochastic-mrcd}) plan for the Gaussian-Poisson demand distribution with parameters $\{(\hat{\mu}_{\ss}^{(\tau)}, \hat{\sigma}_{\ss}^{(\tau)}, \hat{\lambda}_{\ss}^{(\tau)})\}_{s, s'}$, and execute the plan on day $\tau$.
    \State Observe the riders' responses and compute the parameters for the approximate Gaussian posterior $\{(\mu^{\mu,(\tau+1)}_\ss, \sigma^{\mu,(\tau+1)}_\ss, \mu^{\sigma,(\tau+1)}_\ss, \sigma^{\sigma,(\tau+1)}_\ss)\}_{\ss}$. \label{line:TS-8}
\EndFor
%\State{\bf end for}
\end{algorithmic}
\end{algorithm}

\section{Laplace Approximation for Posterior Computation in the Thompson Sampling Algorithm} \label{app:laplace-approximation}

At the end of day $\tau$, for each arc $(s, s')$, suppose $\{p_i, y_i\}_{i \in n_{\ss}^{(\tau)}}$ is the set of prices and rider responses on the arc in history. We compute the likelihood function
\begin{small}
\[
\calL(\mu,\sigma) = \phi\left(\frac{\mu-\mu^\mu_{\ss}}{\sigma^\mu_{\ss}}\right)\phi\left(\frac{\sigma-\mu^\sigma_\ss}{\sigma^\sigma_\ss}\right)\prod_{i}L\left(\frac{p_i-\mu}{\sigma},y_i\right),
\]
\end{small}
where $\phi(\cdot)$ is the probability density function (pdf) of the standard Gaussian and $L(x,y) := \left\{
\begin{small}
\begin{aligned}
     &\Phi(x),&y=0 \\
     &1-\Phi(x),&y=1 
\end{aligned}
\end{small}
\right.$.

We adopt the Laplace approximation method for multi-variate likelihood \cite{wang2018efficient} to approximate $\calL$ by a product of Gaussian distributions of $\mu$ and $\sigma$. We firstly find the mode of $\log\calL$: 
\[
(\tilde{\mu}^\mu,\tilde{\mu}^\sigma) = \arg \max \log \calL(\mu,\sigma).
\]
Then we use the symmetric difference quotient method 
\cite{lax2014calculus} 
to compute the numerical Hessian of $\log\calL (\mu,\sigma)$  at $(\tilde{\mu}^\mu,\tilde{\mu}^\sigma)$ as 
\[
H = \begin{bmatrix}
    H_{11}       & H_{12} \\
    H_{21}       & H_{22} \\
    \end{bmatrix} = \left.
    \begin{bmatrix}
    \frac{\partial^2 \log \mathcal L}{\partial \mu^2}       & \frac{\partial^2 \log \mathcal L}{\partial \mu\partial \sigma} \\
    \frac{\partial^2 \log \mathcal L}{\partial \mu\partial \sigma}       & \frac{\partial^2 \log \mathcal L}{\partial \sigma^2} \\
    \end{bmatrix} \right|_{(\mu,\sigma) = (\tilde{\mu}^\mu,\tilde{\mu}^\sigma)} .
\]
Finally, we set $\mu^{\mu,(\tau+1)}_{\ss}=\tilde{\mu}^\mu$, $\mu^{\sigma,(\tau+1)}_{\ss}=\tilde{\mu}^\sigma$, and 
$\sigma^{\mu,(\tau+1)}_\ss= \sqrt{-H_{11}^{-1}}$, $\sigma^{\sigma,(\tau+1)}_\ss = \sqrt{-H_{22}^{-1}}$ as the parameters for the approximate Gaussian posterior on day $\tau$, as well as the Gaussian prior on day $(\tau+1)$. Note that $H$ is a negative semi-definite matrix and therefore both $-H_{11}$ and $-H_{22}$ are non-negative.

\section{Additional Experiments} \label{app:add-exp}

\subsection{Regularity of the Gaussian-Poisson Distribution}\label{app:add-exp-regularity}

In this part, we perform numerical experiments to show that the edge reward function is concave for Gaussian-Poisson distributions. Without loss of generality we can set $\mu=1$ (up to normalization). We then choose different $(\sigma,\lambda)$ and verify the convexity of edge reward function. We have made a $330\times 330$ grid for $\sigma\in[0,1.5]$ and $\lambda\in[0,33]$, and verified that at all grid points observe the regularity condition. The range of this grid covers the data appearing in the dataset of Section~\ref{sec:experiments} and the resolution is fine, so it empirically verifies the regularity of the Gaussian-Poisson distribution.

In Figure~\ref{fig:marginal-rewards}, we plot the marginal rewards $v_k'$ with $\mu=1$ and a few representative $\sigma$ and $\lambda$ values. It is easy to see that all curves are monotonically non-increasing with $k$.

\begin{figure}[h]
\centering
\includegraphics[width=0.95\columnwidth]{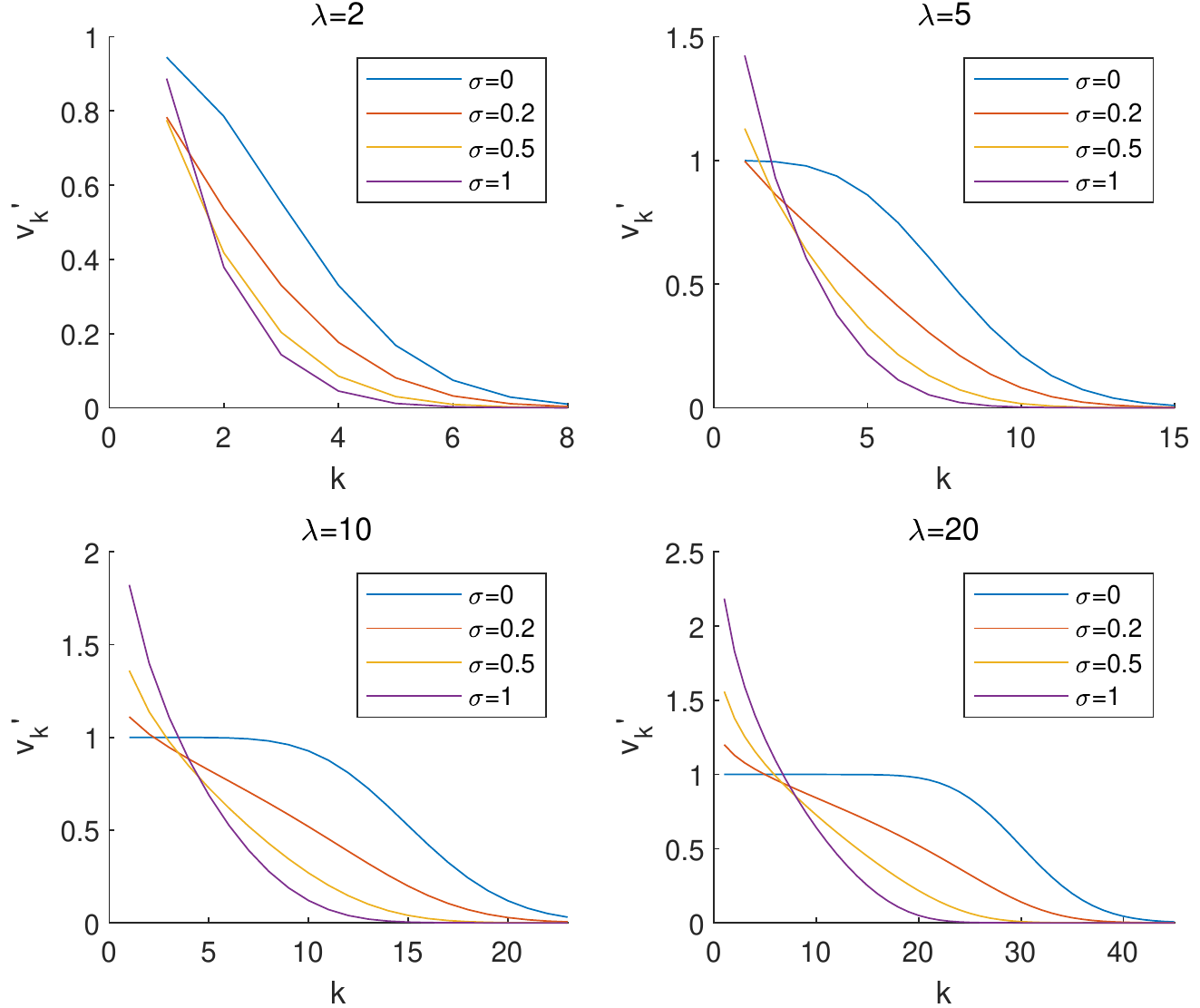} 
\caption{Marginal rewards $v_k'$ with different parameters.}  \label{fig:marginal-rewards}
\end{figure}

\subsection{Experiments for online learning}\label{app:add-exp-robustness}

\vvspace{-1ex}
\paragraph{Online Setting.} When the model parameters are not known before hand, we run our online learning algorithm (in Section~\ref{sec:online-learning}) for 50 days. We refer to the revenue of the algorithm as the \emph{Thompson Sampling value} ({\sf TS}). We also introduce the baseline \emph{exploration-and-exploitation} ({\sf EE}), another common strategy in online learning. In the first $19$ days, {\sf EE} performs exploration where the prices are chosen uniformly in a pre-defined interval, and on day $20$ we learn the model parameters using the first 19-day data, then compute the optimal plan based on the learned parameters for the rest of days.

\vvspace{-1ex}
\paragraph{Results (Online).}

We present the learning curves (the revenue collected on each day) of the online methods in Figure~\ref{fig:exp-learning-curve}. In the figure, we also plot {\sf OV} for reference. We see that {\sf TS} approaches the target {\sf OV} much faster than {\sf EE}.\footnote{To reduce the computational burden, we only update the policy for {\sf TS} in a subset of the 50 days, which results in the observable non-smoothness of the learning curve. If the policy were updated everyday, the performance would be slightly better.} For each online algorithm ${\sf A} \in \{{\sf TS}, {\sf EE}\}$, we define its \emph{average regret} to be
${\sf Reg}({\sf A}) := \frac{1}{50}\sum_{i=1}^{50} ({\sf OV} - {\sf A}(i))$,
where ${\sf A}(i)$ denotes the revenue of {\sf A} on day $i$. ${\sf Reg}({\sf A})$ is a standard metric in online learning that measures the average price paid by {\sf A} on each day to learn and approach the target {\sf OV}. We report that ${\sf Reg}({\sf TS}) = 1.29 \times 10^4$ and ${\sf Reg}({\sf EE}) = 3.38 \times 10^4$. Our {\sf TS} algorithm incurs a much smaller regret than the baseline {\sf EE}.

%We see that the performance of {\sf TS} starts low but surpasses {\sf FP} from the 10th day, yielding a lower regret of {\sf FP}. It also significantly outperforms {\sf EE} during the whole period.

\begin{figure}[h] 
\centering
\includegraphics[width=0.7\columnwidth]{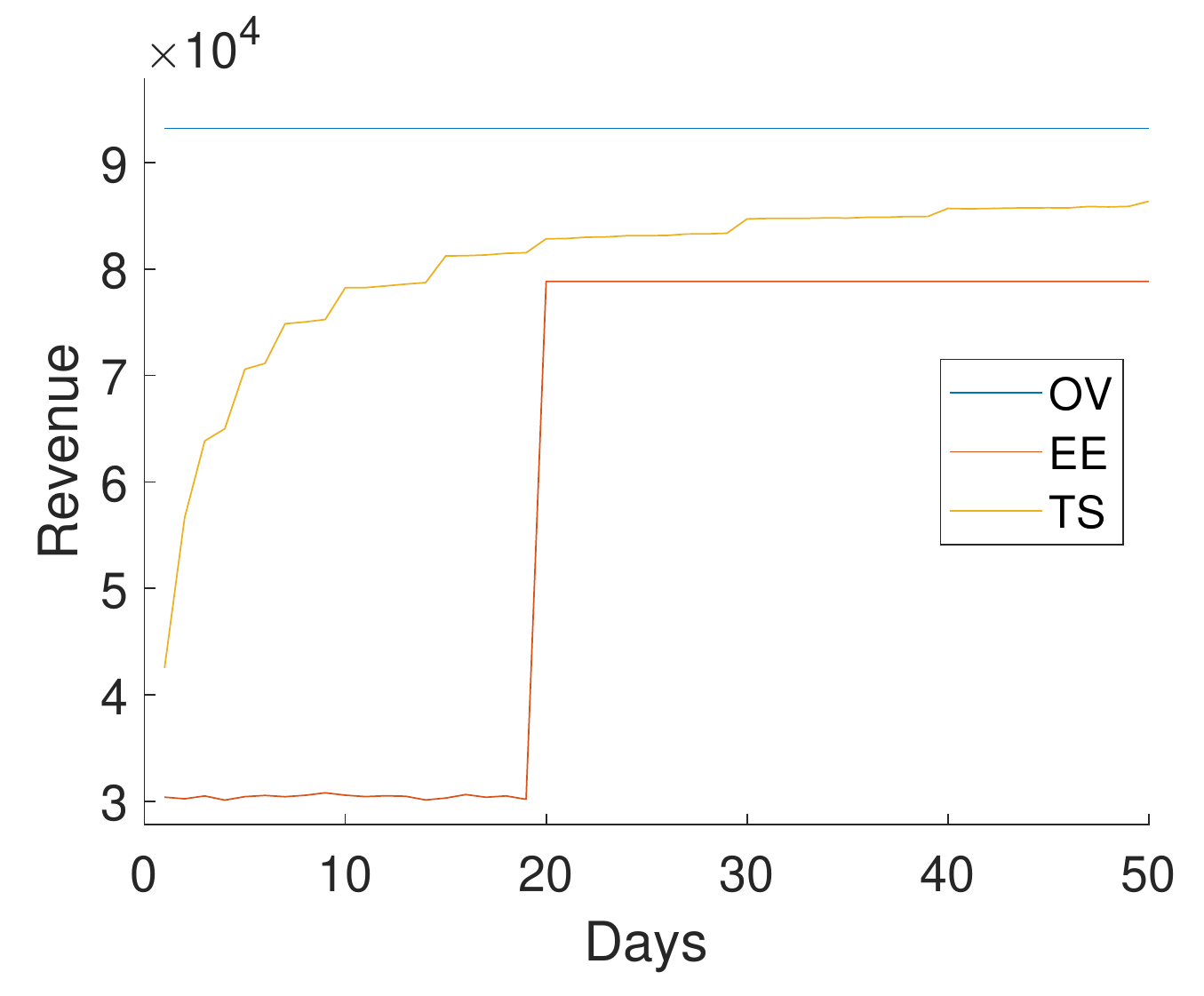} 
\caption{Comparison of learning curves} \label{fig:exp-learning-curve}
\end{figure} 

%\vvspace{-1ex}
\textbf{Robustness.}
To evaluate the generalization ability of our algorithm, we modify the following two key parameters in experiments: the number of drivers and the standard deviations of the riders' valuations. We report the experimental results showing that our algorithms still perform well under these different experimental environments.

In Table~\ref{table:drivers}, we modify the number of drivers. In the $50\%$ drivers setting we remove each driver from the system with $50\%$ independent probability and in the $200\%$ drivers setting we duplicate every driver. In Table~\ref{table:sigma}, we modify the variations of the riders' valuations. Compared to the original dataset, we modify the standard deviations of valuations by $0.5$ and $1.5$ times respectively. We see that in all settings, our {\sf TS} algorithm consistently performs better than other baselines.

For {\sf EE} and {\sf TS}, we present the revenue on the 50th day ({\sf Rev}) and average regrets ({\sf Reg}) during the period. Learning curves of experiments with modified parameters are shown in Figures \ref{fig:modified:l}-\ref{fig:modified:r}.

%In the one-phase estimation we draw prices from a prior distribution for 30 days, then perform MLE after the learning and estimate the revenue with the dispatching and pricing policy from estimated distribution on the actual distribution of traffic. In the adaptive estimation,  we perform the Thompson sampling to estimate the distribution of valuations.

%\comment{EE$_1$: 20/30;~~ EE$_2$: 15/35}
\begin{table}[hbt]
\centering
\begin{tabular}{|c|c|c|c|c|c|c|c|c}
\hline
   \#drivers & \multicolumn{2}{c|}{6655} & \multicolumn{2}{c|}{13411} & \multicolumn{2}{c|}{26822} \\ \hline
   & {\sf Rev} & {\sf Reg} & {\sf Rev} & {\sf Reg} & {\sf Rev} & {\sf Reg} \\ \hline
{\sf OV} & 6.82 & --  & 9.32 &  --  & 11.17 & -- \\ \hline
{\sf FP} &  5.54 &  --  & 7.56 &  --  & 9.02 & -- \\ \hline
{\sf EE} &  5.88 & 2.47  &  7.88 & 3.38   & 9.27 & 4.14 \\ \hline
%{\sf EE}$_2$ &  &   & \\ \hline
{\sf TS} &  \bf 6.40 & \bf 0.87 &  \bf 8.64 & \bf 1.29  & \bf 10.25& \bf 1.65  \\ \hline
\end{tabular}
\caption{\small{{\sf Rev}/{\sf Reg} with different numbers of drivers $(\times 10^4)$}.}
\label{table:drivers}
\medskip
\centering
\begin{tabular}{|c|c|c|c|c|c|c|}
\hline
   stddev    & \multicolumn{2}{c|}{$0.5\sigma$} &  \multicolumn{2}{c|}{$1.0\sigma$} &  \multicolumn{2}{c|}{$1.5\sigma$} \\ \hline
   & {\sf Rev} & {\sf Reg} & {\sf Rev} & {\sf Reg} & {\sf Rev} & {\sf Reg} \\ \hline
{\sf OV}  & 10.36 & --  & 9.32 & --   & 8.61 & -- \\ \hline
{\sf FP} &  7.90 & --  & 7.56 & -- &  7.25& --   \\ \hline
{\sf EE} &  8.41 & 4.11  &  7.88 & 3.38 &  7.45 & 2.90 \\ \hline
%{\sf EE}$_2$ &  &   & \\ \hline
{\sf TS} &  \bf 9.49 & \bf 1.72 &  \bf 8.64 & \bf 1.29   & \bf 7.98 & \bf 1.11  \\  \hline
\end{tabular}
\caption{\small{{\sf Rev}/{\sf Reg} with modified standard deviations $(\times 10^4)$}.}
\label{table:sigma}
\end{table}

\begin{figure}[h] 
\centering
\includegraphics[width=0.7\columnwidth]{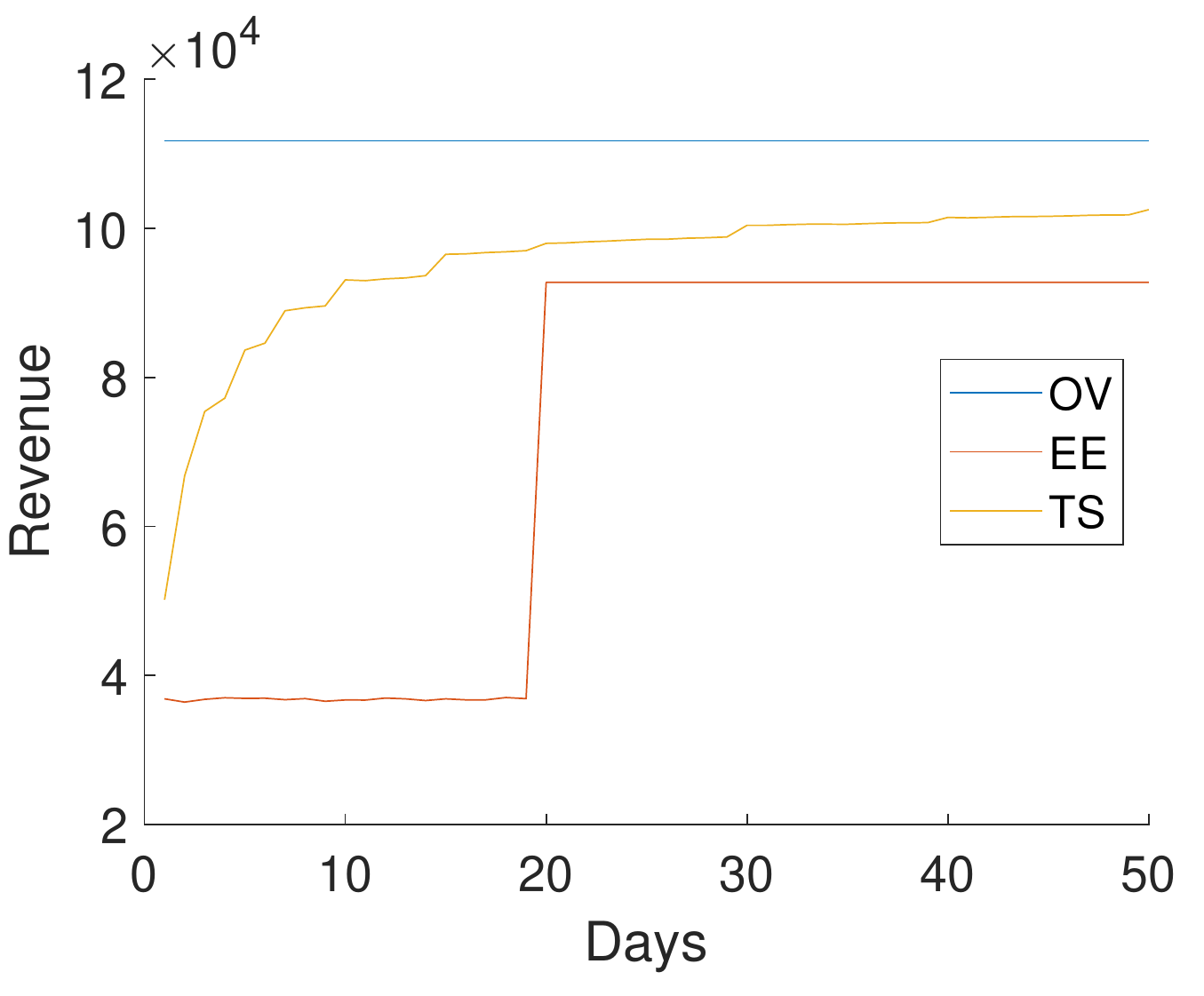} 
\caption{Learning curves with 200\% drivers.} 
\label{fig:modified:l}
\end{figure} 

\begin{figure}[h] 
\centering
\includegraphics[width=0.7\columnwidth]{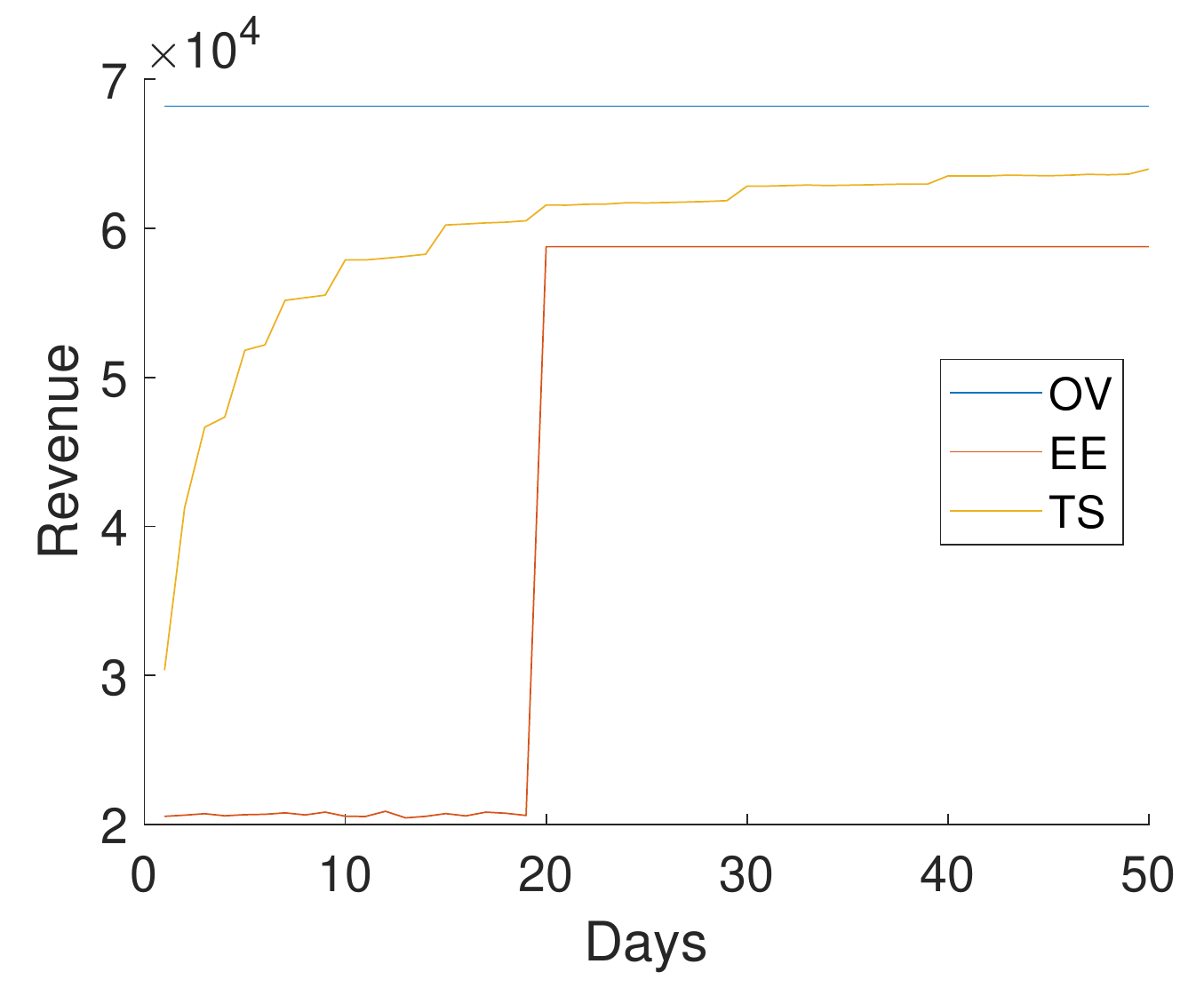} 
\caption{Learning curves with 50\% drivers.} 
\end{figure} 

\begin{figure}[h] 
\centering
\includegraphics[width=0.7\columnwidth]{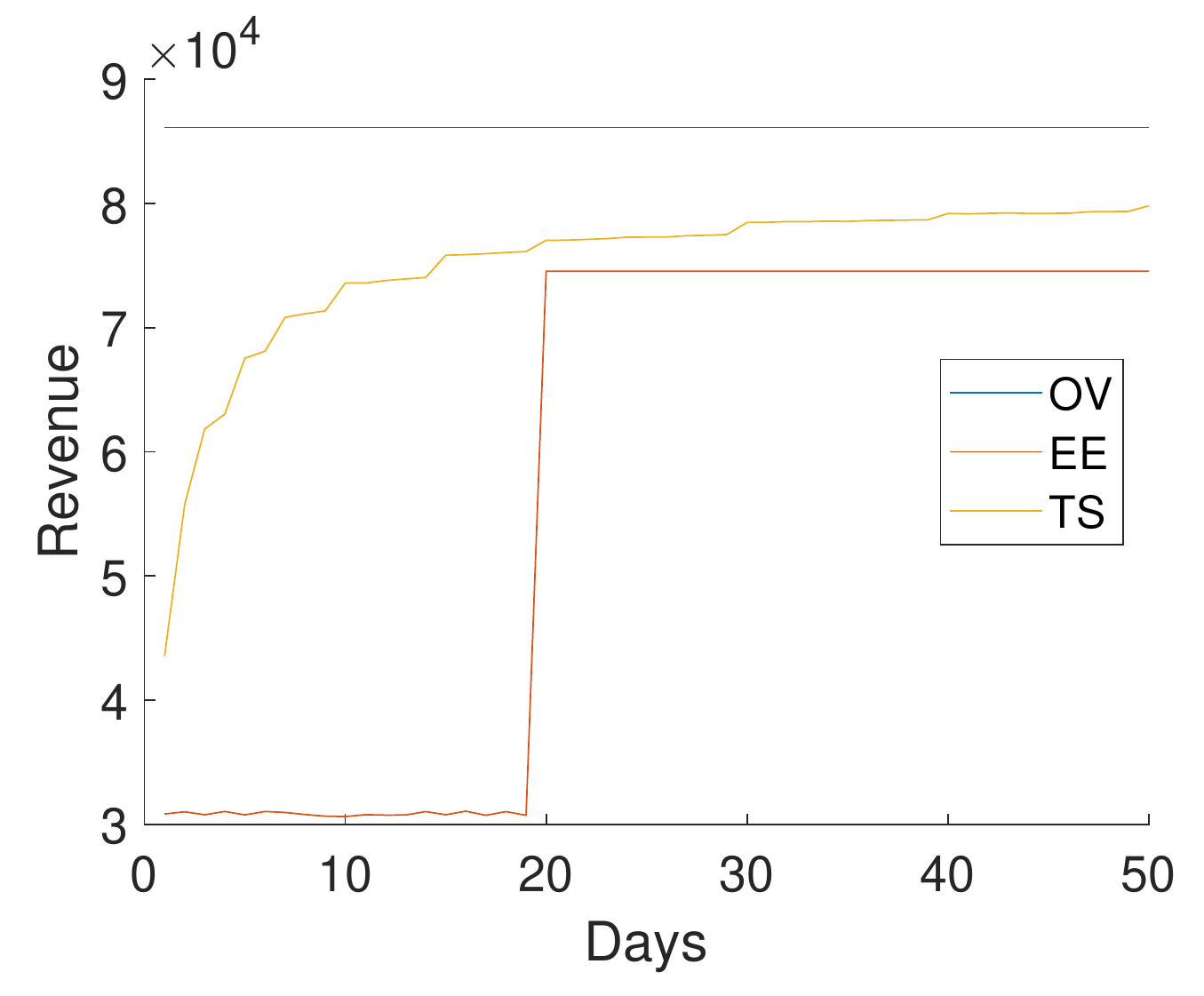} 
\caption{Learning curves with 150\% standard deviations of valuations.} 
\end{figure} 

\begin{figure}[h] 
\centering
\includegraphics[width=0.7\columnwidth]{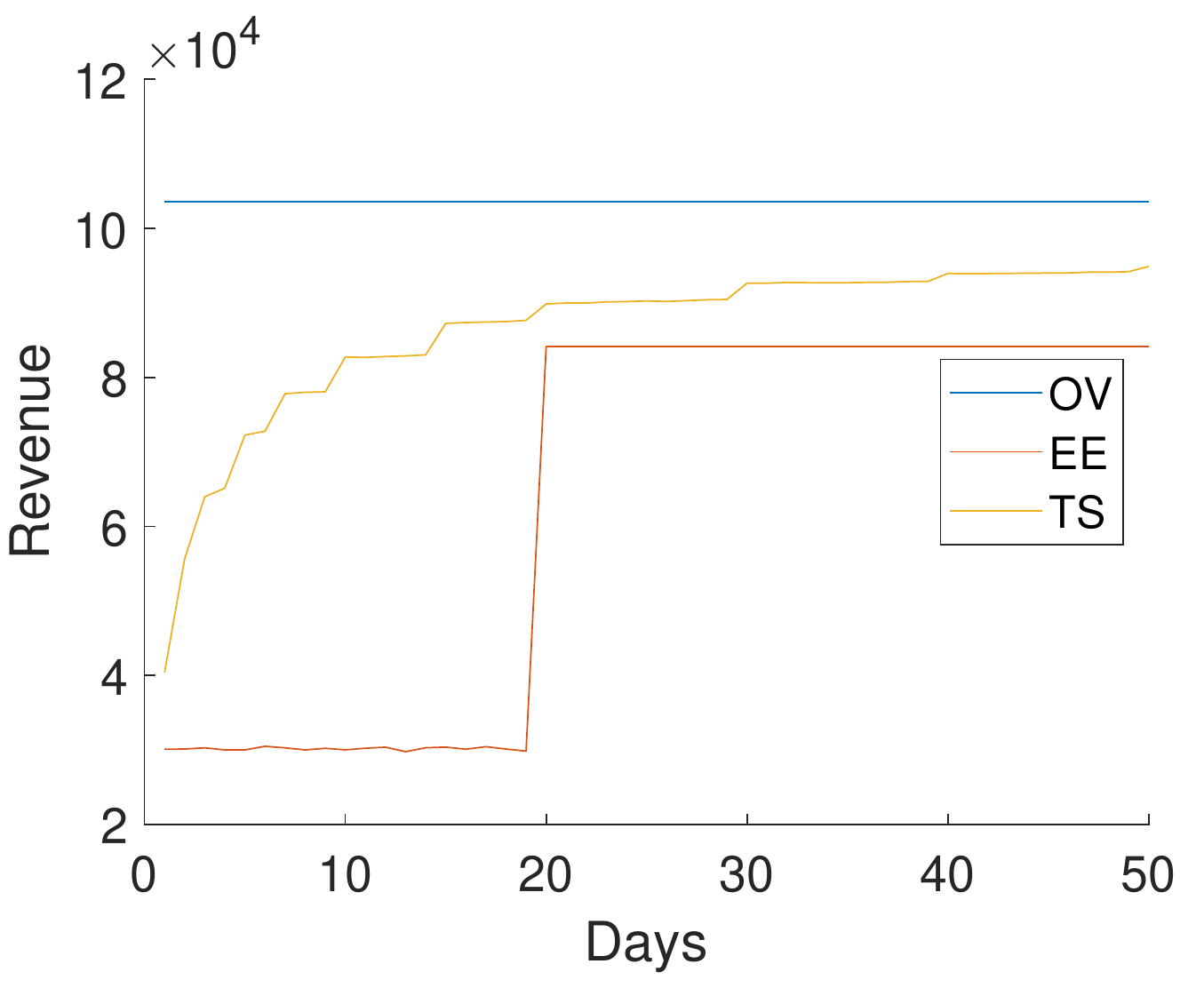} 
\caption{Learning curves with 50\% standard deviations of valuations.}
\label{fig:modified:r}
\end{figure}

\subsection{An Illustrative Example for Fair Re-allocation}\label{app:add-exp-example}

In this part, we show the properties of fair re-allocation for running the Phase 2 algorithm on DiDi dataset in the deterministic setting. Due to the large size of the dataset, we draw a representative subset of the whole dataset to show its behavior, and impose the budget-balance constraint on this subset instead of the whole dataset of rides.

In this example, we consider four positions in Chengdu city in China. Position $A$ is the South Railway Station of Chengdu; position $B$ is Tianfu Square, the leisure and business center located in the center of Chengdu; positions $C$ and $D$ are in two residential districts (Shuangqiaozi and Caojia Alley respectively). We then consider the traces of 10 drivers initiating from $C$, and three consecutive time stamps $1,2,3$ representing the time period of 8:00am to 8:45am. The trip from each position to another takes one time step, but as position $A$ is relatively far from the cluster of $\{B,C,D\}$, trips to or from $A$ typically earn more revenues. Figure~\ref{fig:exp:phase1} shows the numbers of rides and net incomes of each arc from the Maximum Revenue Car Dispatching algorithm.

On the riders' side, riders traveling from or to $A$ are not expected to complain about higher prices, because they do have longer trips. However when it comes to drivers, they may prefer to take longer rides from or to $A$ than traveling among $B,C,D$. Particularly, one driver $J_1$ is assigned the trip $C\to A\to D$ and gains a net income of $6.76$, and another driver $J_2$ is assigned the trip $C\to D\to D$ and gains a net income of $3.71$. Then, $J_2$ may envy $J_1$ for earning more merely because assigned a ``better'' route.

In the same example, after we run the re-allocation algorithm, we re-allocate the money collected from riders to drivers, so that the net utilities for drivers of rides are shown in Figure~\ref{fig:exp:phase2}. In this way, no matter which route is assigned, a driver always gets a total net income of $4.81$ (ignoring rounding errors) within the same total budget, and they cannot improve their net income by deviation, so fairness among drivers are guaranteed while the total revenue is still optimized.

%\begin{figure}[hbt] 
%\centering
%\includegraphics[width=0.95\columnwidth]{pic/map.png} 
%\caption{Map for areas in Appendix %\ref{app:add-exp-example}}
%\label{fig:map}
%\end{figure} 

\begin{figure}[hbt] 
\centering
\includegraphics[width=0.85\columnwidth]{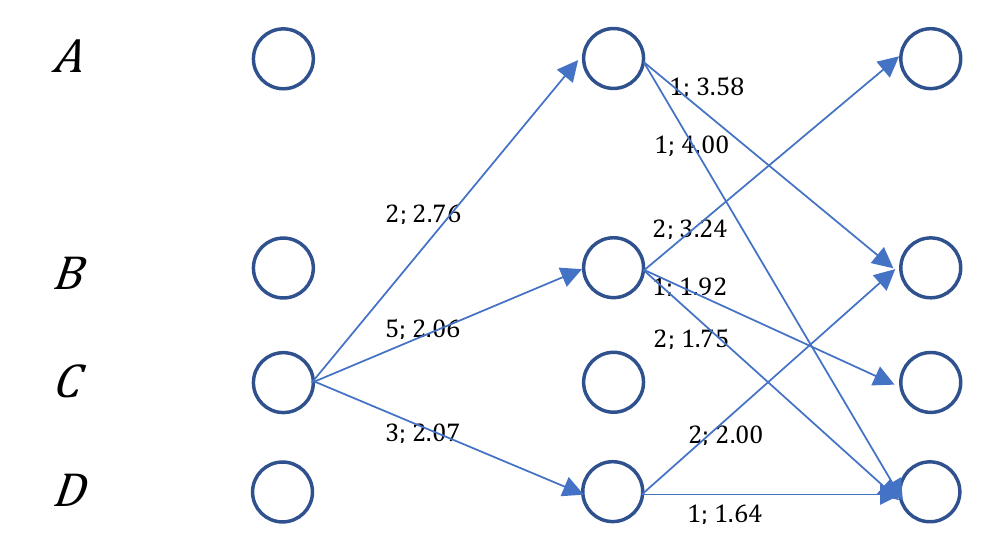} 
\caption{Rider-side pricing for the example in Appendix \ref{app:add-exp-example}} 
\label{fig:exp:phase1}
\end{figure} 

\begin{figure}[hbt] 
\centering
\includegraphics[width=0.85\columnwidth]{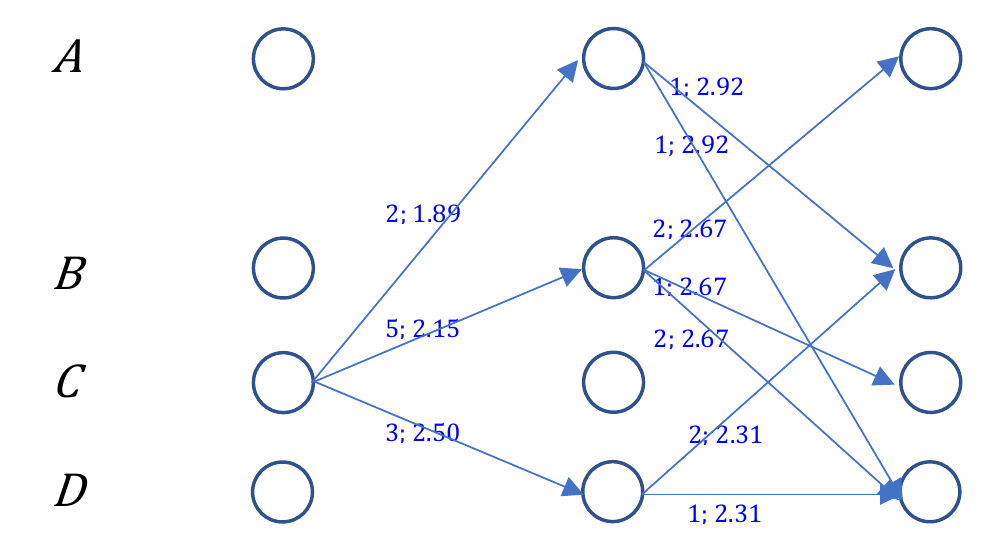} 
\caption{Driver-side re-allocation for the example in Appendix \ref{app:add-exp-example}} 
\label{fig:exp:phase2}
\end{figure} 

\subsection{A simple analysis for influence of number of drivers on unfairness without re-allocation}\label{app:unfairness:drivers}

In fairness evaluation of our experiments, we notice that without the re-allocation phase, the relative unfairness increases with numbers of drivers. Intuitively, when there is only one driver, he/she would just pick the most profitable route; when more drivers join in, if they all choose to pick the most profitable routes for themselves, there may not be enough latent orders for all the routes, and some drivers would have to drive through sub-optimal routes for their income. This phenomenon increases with the number of drivers, which leads to the increase of relative unfairness.

As an simple example, when there are only $5$ latent orders from $A$ to $B,C,D,E,F$, with profits $10,9,8,7,6$ respectively. When there are $2$ drivers initially at $A$, they will be dispatched with $A\to B$ and $A\to C$ trips, and their profits are $10$ and $9$, so the relative unfairness is $0.053$. If there are $5$ drivers instead, then all trips will be taken and there is a wider spread in profits of individual drivers, and the relative unfairness is $0.177$.

However, although it is the general tendency, the relative unfairness is \textbf{not} guaranteed to monotonically increase with the number of drivers. Consider an example in which there are also $45$ latent orders from $A$ to $G$ with profit $2$, then the relative unfairness for $25$ drivers is $0.776$ while the relative unfairness for $50$ drivers is $0.713$. This phenomenon occurs because in the case of $50$ drivers, most drivers can only get the same low profit, so it becomes ``less unfair'' than the case of $25$ drivers.

\end{document}